\documentclass[11pt]{article}

\usepackage{mathrsfs}

\usepackage{latexsym}
\usepackage{bbm}
\usepackage{amsmath}
\usepackage{amsfonts}
\usepackage{amssymb}
\usepackage{multirow}
\usepackage{latexsym}
\usepackage[dvips]{graphicx}
\usepackage{overpic}
\usepackage{graphicx}
\newtheorem{theorem}{\bf Theorem}[section]
\newtheorem{lemma}[theorem]{\bf Lemma}

\newtheorem{remark}{{\bf Remark}}[section]
\newenvironment{proof}{\noindent{\em Proof.}}{\quad \hfill$\Box$\vspace{2ex}}

\def \and {\, \mbox{\rm and}\, }

\def \supp {\,{\rm supp}\,}

\makeatletter

\newcommand{\Rmnum}[1]{\expandafter\@slowromancap\romannumeral #1@}
\makeatother
\begin{document}

\title{\bf  Sharp variational inequalities   for average operators over finite type curves in the plane  }
\author{Xudong Nie\thanks{Department of Mathematics
and Physics, Shijiazhuang Tiedao University, Shijiazhuang 050043, P.R. China. E-mail address:
{\it nxd2016@stdu.edu.cn}.}
}
\date{}
\maketitle

\begin{abstract}
The aim of this article is to establish the $L^p(\mathbb{R}^2)$-boundedness of the
variational operator associated with  averaging operators defined  over finite type curves in the plane.
Additionally, we present  the necessary conditions for the boundedness of these operators in $L^p$.
%These results are nearly optimal.
Furthermore,  to prove one of these results, we establish a mixed-norm local smoothing estimate from
$L^4$ to $L^4(L^2)$ corresponding to a family of Fourier integral operators that do not uniformly satisfy the cinematic curvature condition.
\end{abstract}


AMS Subject Classification (2010) 42B20, 42B35

\vspace*{0.5mm}
%%\begin{equation}\end{equation}
%%\begin{eqnarray}\end{eqnarray}
%%\begin{eqnarray*}\end{eqnarray*}
%%\nonumber\\
%%\begin{defn}\end{defn}
%%\begin{prop}\end{prop}
%%\begin{proof}\end{proof}
%%\begin{theorem}\end{theorem}
%%\begin{coro}\end{coro}
%%\section{}
%%\noindent{}

\section{Introduction}
%The main purpose of this paper is to consider the  variational estimates  for averaging operators defined  over flat curves in the plane.
%We extend the variational inequalities of Jone-Seeger-Wright \cite{2008-strongvariation} for Bourgain-type circular mean operators to the %general case where the Gaussian curvature is allowed to be zero. Additionally, our results can also be viewed as an extension of the relevant %conclusions of Iosevich \cite{1994-Iosevich} and Li \cite{LIwenjuan} regarding  maximal operators over flat curves in $\mathbb{R}^2$.

Suppose  $\overrightarrow{a}=\{a_t\}_{t>0}$ is a family of complex numbers.
The $q$-variation of $\overrightarrow{a}$ is defined as
$$
\|\overrightarrow{a}\|_{v_q}=\sup_{L\in\mathbb{N}}\sup_{0<t_1<t_2<\cdots<t_{L}<\infty}
\left(\sum_{i=1}^{L-1}\left|a_{t_{i+1}}-a_{t_i}\right|^q\right)^{\frac1q},
$$
 for $q\in[1,\infty)$ and replacing the $l^{q}$-sum by a sup when $q=\infty$.
 Let $v_q$ denote the space of functions on $(0,\infty)$ with finite $q$-variation norm $\|\cdot\|_{v_q}$ as defined above.
The variational inequality is an invaluable tool in pointwise convergence problems  because it eliminates the need to establish
pointwise convergence on a dense class, which can be challenging in many cases.
L\'{e}pingle \cite{1976-lepingle-first}  was the first to provide a variance inequality for martingale sequences.
Subsequently, Pisier and Xu \cite{1988-Pisger-xu} simplified the proof of L\'{e}pingle's inequality for martingale sequences.
In 1989, Bourgain \cite{1989-Bourgain-variation} utilized  L\'{e}pingle's findings to investigate the Birchoff ergodic average variance inequality and directly achieved pointwise convergence.
Since then, the issue of boundedness of variation in harmonic analysis, probability theory, and ergodic theory has gained significant attention;
also refer to the papers \cite{2020-GUO-variation}, \cite{Mirek-Stein-variation-bootstrap}, \cite{Mirek-Stein-variation-interpolation}, \cite{Mirek-Stein-variation-jump},
\cite{Mirek-Stein-variation-radontype} as well as their references.

Variational operators associated with families of  spheres is an important  topic in the theory of variations.
Let $\mathcal{A}^{sp}=\{A_t^{sp}\}_{t>0}$  be the family of spherical averaging operators
and define
$$
V_q(\mathcal{A}^{sp}f)(x):=\left\|\{\{A_t^{sp}f(x)\}_{t>0}\}\right\|_{v_q}
$$
where
\begin{equation}\label{exp-sphere-av}
A_t^{sp}f(x)=\int_{S^{d-1}}f(x-ty)d\sigma (y)
\end{equation}
with $d\sigma(y)$
denoting  the normalized  surface measure on the unit sphere
$S^{d-1}$.
A classical result established by Stein \cite{1976-Stein-maximal-sphere} ($d>2$) and Bourgain \cite{1986-Bourgain-maximal-plan} ($d=2$) informs us that the corresponding maximal operator  $M^{sh}f=\sup_{t>0}|A_t^{sh}f|$ is bounded on $L^p(\mathbb{R}^d)$ for $p>\frac{d}{d-1}$.
In their work \cite{2008-strongvariation},
Jone, Seeger, and Wright generalized this result and  demonstrated that
$V_q(\mathcal{A}^{sp})$ is bounded on $L^p(\mathbb{R}^d)$ for all $q>\max\{2,p/d\}$ if  $p>\frac{d}{d-1}$ and the $p$-range is sharp.
Furthermore, they also showed that the $L^p$-boundedness of $V_q(\mathcal{A}^{sp})$  fails when either $q=2$ or $q<p/d$ for $p>2d$.
Recently, Beltran , Oberlin and    Roncal, et al. \cite{2022-Beltran-variation-sphere} proved that, for dimensions greater than  2,
$V_q(\mathcal{A}^{sp})$  maps $L^{p,1}(\mathbb{R}^d)$ to $L^{p,\infty}(\mathbb{R}^d)$
for $q=p/d$ and $p>2d$.
For further information regarding generalization of the spherical maximal operator, readers can referred to the literature such as \cite{2024-liwenjuan-variantion},
\cite{2024-Liu-maxial-sphere}, \cite{2019-lacey-sparse-maximal}, along with their respective references.

The classical results on the spherical maximal operator have initiated  investigations into various
classes of maximal operators  associated with hypersurfaces where the Gaussian curvature
 is allowed to vanish.
Let $\eta\in C_{0}^{\infty}(S) $ be a smooth non-negative function with compact support on a smooth hypersurface $S\subset\mathbb{R}^d$.
Then the associated maximal operator is defined as
\begin{equation*}\label{exp-hyper}
M^{hyper}f(x)=\sup_{t>0}\left|A_t^{hyper}f(x)\right|,
\end{equation*}
where
\begin{equation*}\label{exp-hyper}
A_t^{hyper}f(x)=\int_{S}f(x-ty)\eta(y)d\sigma (y).
\end{equation*}
The boundedness index, denoted as $p(S)$, is determined for each hypersurface $S$ and a fixed density function $\eta(x)$. It represents the smallest value of $p$ for which the operator $M^{hyper}$ remains bounded in the space $L^p$. This index is closely linked to the geometric characteristics of the hypersurface. Previous research has successfully determined the precise value of $p(S)$ for various classes of hypersurfaces with values greater than or equal to 2. Furthermore, in \cite{2019-Buschen-part1,2019-Buschen-part2}, researchers have obtained the exact boundedness index for smooth hypersurfaces of finite type in three-dimensional Euclidean space that satisfy a natural transversality condition (i.e., no affine tangent plane passing through any point on $S$ intersects with the origin in $\mathbb{R}^3$) when $p(S)<2$.
The value of $p(S)$ is connected to the oscillation index of an oscillatory integral that arises from the Fourier transform of the surface measure $\eta dS$, as well as the contact index of the hypersurface. Relevant literature on this topic can be found in \cite{2010-Ikromov-muller,1997-Iosevich}, and further works are available in \cite{1985-Sogge-Stein} and \cite{1995-Sogge}.

While there has been notable progress in studying the boundedness of maximal operators linked to hypersurfaces, research on variational operators corresponding to these surfaces remains somewhat unsatisfactory. Taking inspiration from findings related to spherical variational operators, this paper aims to establish $L^p$-inequalities for variational operators associated with curves, even when the Gaussian curvatures at certain points are allowed to vanish within the plane.

%It should be noted that our conclusion seems to imply a significant relationship between  maximal operators and the variational operators associated to
%hypersurface in general dimension
%as follows.\\
%{\bf Conjecture} \textit{
%Suppose
%$$q(S) :=\inf\{q :\mbox{there exists $p$ such that the operator   $V_q(\mathcal{A}^{hyper})$ is bounded in $L^p$}\},$$
 %where $\mathcal{A}^{hyper}=\{A_t^{hyper}\}_{t>0}$. Then $q(S)=\max\{2,p(S)\}$ and  a priori estimates
%\begin{equation}\label{Mainresult01-1}
%\left\|V_q(\mathcal{A}f)\right\|_{L^p(\mathbb{R}^2)}  \leq C_{p,q} \|f\|_{L^p(\mathbb{R}^2)}
%\end{equation}
%holds whenever  $q>\max\{q(S),p/d\}$ if  $p>p(S)$.
%Conversely, if (\ref{Mainresult01-1}) holds, then we necessarily have $q\geq p/d$.
%}

\subsection{Statement of main theorems in the plane}

Consider a smooth curve denoted by $C:I\rightarrow\mathbb{R}^2$ in the plane.
Define $A_tf(x)=\int_{C}f(x-ty)d\sigma(y)$,
where $I$ represents a compact interval in $\mathbb{R}$ and
$d\sigma(y)$ represents the product of a cutoff function with Lebesgue measure on $C$.
We say that $C$ is finite type if $\langle(C(x)-C(\bar{x})),\vec{a}\rangle$
never  vanishes  infinitely  for any $\bar{x}\in I$ and any unite vector $\vec{a}$.
Additionally, suppose $C$ takes form $(x_1,\gamma(x_1))$ in a small neighborhood of $a_0$.
As linear transformation preserves $L^p$-boundedness for both maximal and variational operators,
we can choose $\gamma(x_1)$ such that $\gamma'(a_0)=0$.
If $\gamma^{(k)}(a_0)=0$ holds true for all $k$ from $1$ to $m-1$ while $\gamma^{(m)}(a_0)\neq0$,
then we classify $C$ as having finite type $m$ at $a_0$.

We need to consider two distinct categories of curves. Suppose that the tangent line at a point on the curve does not intersect with the origin. If the Gaussian curvature of $C$ is non-zero, then the corresponding maximal operator exhibits boundedness on $L^p$ for values of $p$ greater than 2, as mentioned in \cite{1986-Bourgain-maximal-plan}. In cases where the curvature may vanish at certain points and $C$ possesses finite type of at least $m$ at any given point, then the corresponding operator demonstrates boundedness on $L^p$ if and only if $p>m$, as discussed in \cite{1994-Iosevich}.

On the contrary, if the tangent line at a point on the curve passes through the origin, it is implicitly suggested in Iosevich's article \cite{1994-Iosevich} that the corresponding maximal operator is bounded on $L^p$ for $p>1$.
However, this issue becomes more complex when dealing with non-isotropic dilations.

 Without loss of generality, let us assume that $\gamma(x_1)=(x_1-a_0)^{m}\mathcal{O}(1+(x_1-a_0)^n)$,
if $1\leq n<\infty$, Li  \cite{LIwenjuan} demonstrated that the maximal operator
\begin{equation*}\label{liwen-maximal}
\mathcal{M}f(x)=\sup_{t>0}\left|\int_{C}f(x_1-ty_1,x_2-t^my_2)d\sigma(y)\right|
\end{equation*}
maps $L^p$ to $L^p$ if and only if $p>2$.
In case $a_0=0$ and $n=0$, Stein \cite{1976-stein-curves} proved that $\mathcal{M}$ is bounded on $L^p$ for $p>1$.

Given that maximal operators can be regarded as a specific instance of variational operators, it is pertinent to inquire whether these aforementioned results can be generalized to the corresponding variational operators. Herein, we present our principal finding in the two-dimensional plane.

Let $b\in C^{\infty}(I,\mathbb{R})$,
where $I$ is a bounded interval containing the origin, and
\begin{equation}  \label{condition-b}
b(0)\neq0;b^{(j)}(0)=0,j=1,2,\cdots,n-1;b^{(n)}(0)\neq0\,\,(n\geq1).
\end{equation}
Suppose $\phi(y)=a_0+y^mb(y)$, where $m\geq2$.
\begin{theorem}[isotropic dilations] \label{Th-isotropic}
Define the average operaters $\mathcal{A}=\{A_t\}_{t>0}$, where
 \begin{equation}
A_tf(x)=\int_{\mathbb{R}}f(x_1-ty,x_2-t\phi(y))\eta(y)dy,
\end{equation}
with $\eta(y)$  supported in a sufficiently small neighborhood of the origin.
Suppose  $b$ satisfies (\ref{condition-b}).

(i) If $a_0\neq0$,
then a priori estimates
\begin{equation}\label{Mainresult01}
\left\|V_q(\mathcal{A}f)\right\|_{L^p(\mathbb{R}^2)}  \leq C_{p,q} \|f\|_{L^p(\mathbb{R}^2)}
\end{equation}
holds whenever $q>\max\{m,p/2\}$ and $p>m$.
Conversely, if (\ref{Mainresult01}) holds, then we necessarily have $q\geq p/2$.

(ii) If $a_0=0$,
then (\ref{Mainresult01})
holds whenever $q>\max\{2,p/2\}$ and $p>1$.
Conversely, if (\ref{Mainresult01}) holds, then we necessarily have $q\geq p/2$.
\end{theorem}
\begin{remark}
Since $Mf(x)\leq V_q(\mathcal{A}f)(x)+|f(x)|$, inequality (i) is a generalization of Iosevich's result \cite{1994-Iosevich} in the plane.
\end{remark}
\begin{remark}
When $m=2$, the Gaussian curvature dose not vanish. So (i) is a generalization of the result in  Jone-Seeger-Wright \cite{2008-strongvariation} for Bourgain-type circular mean operators for $d=2$.
\end{remark}
%\begin{remark}
%By this theorem, we see that $q(C)\leq m$ in  (i), if we could prove that $q(C)=m$, then our conjecture will valid in the plane.
%However, this problem is still open.
%\end{remark}

\begin{theorem}[nonisotropic dilations]  \label{Th-nonisotropic}
Define the average operaters $\mathcal{B}=\{B_t\}_{t>0}$, where
\begin{equation}
B_tf(x)=\int_{\mathbb{R}}f(x_1-ty,x_2-t^my^mb(y))\eta(y)dy,
\end{equation}
where $\eta(y)$ is supported in a sufficiently small neighborhood of the origin.
Suppose $b$ satisfies (\ref{condition-b}).
Then a priori estimates
\begin{equation}\label{Mainresult01-2}
\left\|V_q(\mathcal{B}f)\right\|_{L^p(\mathbb{R}^2)}  \leq C_{p,q} \|f\|_{L^p(\mathbb{R}^2)}
\end{equation}
holds
whenever $q>\max\{2,p/2\}$ and $p>2$.
Conversely, if (\ref{Mainresult01-2}) holds, then we necessarily have $q\geq p/2$.
\end{theorem}

\begin{remark}
Since $\mathcal{M}f(x)\leq V_q(\mathcal{B}f)(x)+|f(x)|$, inequality (\ref{Mainresult01-2}) is a generalization of Li's result \cite{LIwenjuan} in the plane.
\end{remark}

\subsection{Organization of the article}

In this study, we frequently employ the long-short variational decomposition method, the stationary phase method, and Sobolev embedding inequality to address relevant issues. Section 2 will focus on presenting these three areas of expertise. In section 3, we establish the $L^p$-boundedness of variational operators associated with isotropic dilations of curves in the plane. In section 4, we demonstrate the $L^p$-boundedness of variational operators related to curves of finite type $m$ ($m\geq2$) and their associated dilations $(x_1,x_2)\rightarrow (tx_1,t^mx_2)$.
A crucial step in supporting our aforementioned conclusion involves estimating the $L^4_x(L^2_t)$-bounds for Fourier integral operators. While the relevant Fourier integral operator adheres uniformly to the cinematic curvature condition in case of isotropic dilations, it does not satisfy this condition uniformly for non-isotropic dilations. Therefore, it is imperative to extend local integrability results applicable to classical Fourier integral operators in order to encompass this scenario. In section 5, we will undertake this task.

\textit{Conventions}
Throughout this article, we will use the notation $A\ll B$ to indicate that there exists a sufficiently large constant $G$, independent of both $A$ and $B$, such that $GA\leq B$. We will write $A\lesssim B$ to mean that there is a constant $C$, independent of both $A$ and $B$, such that $A\leq CB$. Finally, we will write $A\approx B$ if   $A \lesssim B$ and  $B \lesssim A$.

\section{Preliminaries}

\subsection{Long and short variations}
Firstly, we introduce  the definitions of the long variation operator and short variation operator.
For $q\in[2,\infty)$   and each $j\in\mathbb{Z}$,
let
$$
V_{q,j}(\mathcal{T}f)(x):=\sup_{N\in\mathbb{N}}\sup_{2^j\leq t_1<\cdots<t_{N}\leq2^{j+1}}
\left(\sum_{i=1}^{N-1}\left|T_{t_{i+1}}f(x)-T_{t_i}f(x)\right|^q\right)^{\frac1q} .
$$
We define the short variation opetator
$$
V_q^{sh} (\mathcal{T}f) (x):=\left(\sum_{j\in\mathbb{Z}}\left[V_{q,j}(\mathcal{T}f)(x)\right]^q\right)^{\frac1q} $$
and the long variation operator
$$
V_q^{dyad}(\mathcal{T}f)(x):=\sup_{N\in\mathbb{N}}\sup_{\substack{t_1<\cdots<t_N\\\{t_i\}_{i=1}^{N}\subset\mathbb{N}}}
\left\{\sum_{i=1}^{N-1}\left|T_{2^{t_{i+1}}}f(x)-T_{2^{t_{i}}}f(x)\right|^q \right\}^{\frac1q}.
$$

Based on the following lemma (see \cite[lemma 1.3]{2008-strongvariation}),
the estimate of $V_q(\mathcal{T})$  is split into considering  the long variation operator and the short variation operator.
\begin{lemma}\label{lemma-long-short}
For $q>2$,
\begin{equation}\label{ineq-long-short} V_q(\mathcal{T}f)(x)\lesssim   V_q^{sh}(\mathcal{T}f)(x) + V_q^{dyad}(\mathcal{T}f)(x). \end{equation}
\end{lemma}
 \begin{lemma}(see \cite[Theorem 1.1]{2008-strongvariation}) \label{lemma-long-case}
 Suppose $\sigma$ is a compactly supported finite Borel measure on $\mathbb{R}^d$,
 and let $\mathcal{U}=\{U_t\}_{t>0}$ where
 $$U_tf(x)=\int_{\mathbb{R}^d}f(x-t^Py)d\sigma(y),$$
 with $P$ being a real $2\times2$ matrix whose eigenvalues having positive real parts.
 If there exists a number $b>0$ such that $|\widehat{\sigma}(\xi)|\leq C|\xi|^{-b}$,
 then
 $$
 \left\|V_q^{dyad}(\mathcal{U}f) \right\|_{L^p(\mathbb{R}^d)}   \leq C\|f\|_{L^p(\mathbb{R}^d)} $$
 holds for $1<p<\infty$ and $q>2$.
 \end{lemma}
 \begin{lemma} (see \cite[Lemma 6.1]{2008-strongvariation}) \label{lemma-short-remainder}
 Let $\mathcal{U}=\{U_t\}_{t>0}$,  where $U_tf=f*\sigma_t$ is described above.
 Suppose that $|\widehat{\sigma}(\xi)|\leq C_b (1+|\xi|)^{-b}$ for some $b>\frac{d+1}{2}$ .
 Then the a priori estimate
 $$
 \left\|V_2^{sh}(\mathcal{U}f)\right\|_{L^p(\mathbb{R}^d)} \leq C_p\|f\|_{L^p(\mathbb{R}^d)} $$
 holds for $p>1$.
 \end{lemma}
\subsection{Method of stationary phase}
The next lemma implies that we can only consider the short variational operators.
 \begin{lemma} (van der corput's lemma)\label{lemma-var der}
 If $\Phi^{(i)}(0)=0$  for $0\leq i\leq m-1$ but $\Phi^{(m)}(0)\neq0$ and $m\geq 2$   then
 $$
 \left|\int_0^{\infty}e^{i\lambda \Phi(y)}\eta(y)dy\right|\leq C\lambda^{-1/m}
 $$
 provided that $\eta$ has small enough support.
 \end{lemma}

We frequently employ the subsequent approach of stationary phase.
\begin{lemma}\label{stationary phase}(\cite[Theorem 1.2.1]{1993-sogge-stionary phase})
Let $S$ be s smooth hypersurface in $\mathbb{R}^d$  with non-vanishing Gaussian
curvature and $d\mu$ the Lebesgue measure on $S$. Then
$$
\left|\widehat{d\mu}(\xi)\right|\leq C(1+|\xi|)^{-\frac{d-1}{2}}.
$$
Moreover, suppose that $\Gamma\subset\mathbb{R}^d\backslash 0 $   is the cone consisting of all $\xi$
which are normal to $S$  at some point $x\in S$  belonging to a fixed relatively compact
neighborhood $\mathcal{N}$  of supp $d\mu$.
Then
$$
\left|\left(\frac{\partial}{\partial\xi}\right)^{\alpha}\widehat{d\mu}(\xi)\right|
=\mathcal{O} (1+|\xi|)^{-N}  \,\mbox{for all $N\in \mathbb{N}$, if $\xi\notin \Gamma$},
$$
$$
\widehat{d\mu}(\xi)=\sum  e^{-i\langle x_j,\xi\rangle} a_j(\xi)    \,\,\,\mbox{if $\xi\in \Gamma$},
$$
where the finite sum is taken over all $x_j\in\mathcal{N}$ having $\xi$   as the normal and
$$
\left|\left(\frac{\partial}{\partial \xi}\right)^{\alpha}a_j(\xi)\right|
\leq C_{\alpha}  (1+|\xi|)^{-\frac{d-1}{2}-|\alpha|}.
$$
\end{lemma}

\subsection{Sobolev embedding inequality}
We present a Sobolev embedding inequality that can be utilized for estimating the $q$-variations, along with a technique for estimating the upper bound of $L^4$ for the global square function.
\begin{lemma}(see \cite[inequality (38)]{2008-strongvariation})\label{Sobolev embbding}
 Suppose $F\in C^{1}(I)$, where $I$ is a closed interval in $ \mathbb{R}^+$.
 Then for $q\geq1$,
 $$
 \left\|\{F(t)\}_{t>0} \right\|_{v_q}   \leq C_q\|F\|^{1/q'}_{L^q(I)}\|F'\|^{1/q}_{L^q(I)}, $$
 where $1/q+1/q'=1$.
\end{lemma}

\begin{lemma}(see \cite[p.6737]{2008-strongvariation})\label{multiplier-Th}
Let $\{m_s:1\leq s\leq 2\}$  be a family Fourier multipliers on $\mathbb{R}^d$,
each of which is compactly supported on $\{\xi:1\leq |\xi|\leq 2\}$
and satisfies
$$
\sup_{s\in[1,2]}\left|\partial_{\xi}^{\tau}m_{s}(\xi)\right|\leq B
$$
for each $0\leq|\tau|\leq d+1$ for some positive constant $B$.
Assume that there exists some positive constant $A$ such that
$$
\sup_{j\in \mathbb{Z}}\left\|\left(\int_{1}^{2}|\mathcal{F}^{-1}[m_s(2^j\cdot)\widehat{f}(\cdot)]|^2ds \right)^{\frac12}\right\|_{L^2(\mathbb{R}^d)}
\leq A \|f\|_{L^2(\mathbb{R}^d)}
$$
and
$$
\sup_{j\in \mathbb{Z}}\left\|\left(\int_{1}^{2}|\mathcal{F}^{-1}[m_s(2^j\cdot)\widehat{f}(\cdot)]|^2ds \right)^{\frac12}\right\|_{L^p(\mathbb{R}^d)}
\leq A \|f\|_{L^p(\mathbb{R}^d)}
$$
for some $p>2$.
Then
$$
\left\|\left(\sum_{j\in\mathbb{Z}}\int_{1}^2\left|\mathcal{F}^{-1}[m_s(2^j\cdot)\widehat{f}(\cdot)]\right|^2ds\right)^{\frac12}  \right\|_{L^p(\mathbb{R}^d)}
\lesssim A\log (2+B/A) \|f\|_{L^p(\mathbb{R}^d)}.
$$ \end{lemma}

\section{The proof for variational operators associated with isotropic dilations of curves in the plane}
In this section, we will give the proof of Theorem \ref{Th-isotropic}.
\subsection{The proof of the sufficient condition when $a_0\neq 0$}
We can decompose the variational operator $V_q(\mathcal{A}f)$ into two parts: the short-variational operator $V_q^{sh}(\mathcal{A}f)$ and the long-variational operator $V_q^{dyad}(\mathcal{A}f)$. This division is based on inequality (\ref{ineq-long-short}). Let us define $\widehat{d\mu}(\xi)$ as $$\widehat{d\mu}(\xi)=\int_{\mathbb{R}}e^{-i(y\xi_1+\phi(y)\xi_2)}\eta(y)dy.$$ By applying van der corput's Lemma \ref{lemma-var der}, we can establish that $|\widehat{d\mu}(\xi)|$ is bounded by $C(1+|\xi|)^{-1/m}$, which allows us to derive the $L^p$ inequalities for  $V_q^{dyad}(\mathcal{A}f)$ using Lemma \ref {lemma-long-case}. Therefore, our focus now shifts to analyzing the short-variation component, namely  $V_q^{sh } (\mathcal { A } f )$.

We select a sufficiently small positive value for $B$ so that $(0,a_0)$ is the only flat point of the curve $(y,\phi(y))$ within the interval $(-2B, 2B)$.
Let $\rho_0>0$ be a function supported in $\{x:B/2\leq |x|\leq 2B\}$ and satisfying $\sum_{k}\rho_0(2^ky)=1$ for $y\in\mathbb{R}$.
 Define $\mathcal{A}^k=\{A_t^k\}_{t>0}$, where
$$A_t^kf(x):=\int_{\mathbb{R}} f(x_1-ty,x_2-t(2^{km}a_0+y^mb(2^{-k}y)))\rho_0(y)\eta(2^{-k}y)dy.$$
Then,
\begin{eqnarray*}
A_t(f)(x)&=&\sum_{k}\int_{\mathbb{R}} f(x_1-ty,x_2-t\phi(y))\eta(y) \rho_0(2^ky)dy\\
&=& \sum_{k}2^{-k}\int_{\mathbb{R}} f(x_1-t2^{-k}y,x_2-t(a_0+2^{-km}y^mb(2^{-k}y)))\eta(2^{-k}y) \rho_0(y)dy \\
&=& \sum_{k} 2^{-k} T_k^{-1}A_t^kT_kf(x) ,
\end{eqnarray*}
where $T_kf(x_1,x_2)=2^{-\frac{(m+1)k}{p}}f(2^{-k}x_1,2^{-mk}x_2)$.
Note that $\|T_kf\|_{L^p}=\|f\|_{L^p}$.
Therefore, it suffices to  prove the following estimate
$$
\sum_{k}2^{-k}\left\|V_q^{sh}(\mathcal{A}^k)\right\|_{L^p\rightarrow L^p}     \leq C_{p,q}
$$
for  $q>\max\{m,p/2\}$ and $p>m$.

Let  $s=-\frac{\xi_1}{\xi_2}$ and  $\delta=2^{-k}$.
By using the Fourier inversion formula, we have
\begin{eqnarray*}
A_t^k(f)(x)&=&\frac{1}{(2\pi)^2}\int_{\mathbb{R}^2}e^{ix\cdot\xi}
\int_{\mathbb{R}} e^{-it
\left(y\xi_1+(2^{km}a_0+y^mb(2^{-k}y))\xi_2\right)}\eta(2^{-k}y)\rho_0(y)dy\widehat{f}(\xi)d\xi\\
&=& \frac{1}{(2\pi)^2}\int_{\mathbb{R}^2}e^{ix\cdot\xi}\widehat{d\mu_k}(t\xi)\widehat{f}(\xi)d\xi,
\end{eqnarray*}
where
$$
\widehat{d\mu_k}(\xi):=\int_{\mathbb{R}}
e^{-i\xi_2\Phi(y,s,\delta)}\eta(2^{-k}y)\rho_0(y)dy,
$$
with $\Phi(y,s,\delta)=-ys+2^{km}a_0+y^mb(\delta y)$.
Then we have
\begin{equation}  \label{001-2}
\left(\frac{\partial }{\partial y}\right)^2\Phi(y,s,\delta) = m(m-1)y^{m-2}\left(b(\delta y)+2m\delta y b'(\delta y)
+\delta^2y^2b''(\delta y)\right).
\end{equation}
Note that $b(0) \neq0 $ and we can assume $\delta$ to be  sufficiently small.
Hence, $\left(\frac{\partial }{\partial y}\right)^2\Phi(y,s,\delta)\neq0$  for $y\in (B/2,2B)$. This result  implies that  there exists a smooth solution $\widetilde{y}=\widetilde{y}(s,\delta)$ satisfying
\begin{equation}  \label{001}
\left(\frac{\partial }{\partial y}\right)\Phi(y,s,\delta) = -s+my^{m-1}b(\delta y)+\delta y^mb'(\delta y)=0.
\end{equation}

Without loss of generality, let's  assume that $b(0)=1/m$. It follows that  $\widetilde{y}(s,0)=s^{\frac{1}{m-1}}$.
Denote $\widetilde{\Phi}(s,\delta)=\Phi\left(\widetilde{y}(s,\delta),s,\delta\right)$. By Taylor expansion about $\delta$,
 we have
\begin{equation}\label{express-Phi}
\widetilde{\Phi}(s,\delta)=-\frac{m-1}{m}s^{\frac{m}{m-1}}+2^{km}a_0+\delta R(s,\delta),
\end{equation}
where
$R(s,\delta)$ is homogeneous of degree zero in $\xi$.
By means of  the method of stationary phase, we have
\begin{equation}\label{main-fenjie}
\widehat{d\mu_k}(\xi)=e^{-i\xi_2\widetilde{\Phi}(s,\delta)}\chi_{k}
\left(s\right)  \frac{A_k(\xi)}{(1+|\xi|)^{\frac12}}  +A_k^{err}(\xi),
 \end{equation}
where $\chi_k$    is a smooth function supported in the interval $[c_k,\widetilde{c}_k]$,
for certain non-zero positive constants $c_1\leq c_k<\widetilde{c}_k\leq c_2$
deponding only on $k$.
$\{A_k(\xi)\}_k$ is contained in a bounded subset of symbols of order zero.
More precisely,
\begin{equation} \label{estimate-mainterm}
|D_{\xi}^{\alpha} A_k(\xi)|\leq C_{\alpha}  (1+|\xi|) ^{-|\alpha|}
\end{equation}
where $C_{\alpha}$ do not depend on $k$.
Furthermore, $A_k^{err}$ is a remainder term and satisfies
$$
 |D_{\xi}^{\alpha} A_k^{err}(\xi)|\leq C_{\alpha,N}  (1+|\xi|) ^{-N},
$$
where $C_{\alpha,N}$   are constants do not depend on $k$.

First, let us consider the remainder part of  (\ref{main-fenjie}).
Set  $\mathcal{A}_k^{err}=\{A_{k,t}^{err}\}_{t>0}$, where $$
A_{k,t}^{err}f(x)=\int_{\mathbb{R}^2} e^{i\xi\cdot x} A_k^{err} (t\xi) \widehat{f}(\xi)d\xi. $$
By Lemma \ref{lemma-short-remainder}, it is easy to get
$$
\|V^{sh}_q(\mathcal{A}_k^{err}f) \|_{L^p(\mathbb{R}^2)} \leq C_{p} \|f\|_{L^p(\mathbb{R}^2)}
$$
for $q\geq2$ and $p>1$.
Thus it is remained to prove
$$
\sum_{k}2^{-k}\left\|V_q^{sh}(\mathcal{A}_{main}^k)\right\|_{L^p\rightarrow L^p}     \leq C_{p,q}
$$
  for $q>\max\{m,p/2\}$ and $p>m$,
where $\mathcal{A}_{main}^{k}=\{A_{main,t}^k\}_{t>0}$,  with
$$A_{main,t}^kf(x):=\int_{\mathbb{R}^2}e^{i\xi\cdot x}A_{main}^k(t\xi)\widehat{f}(\xi)d\xi$$ and
\begin{equation} \label{express-mainterm}
 A_{main}^k(\xi) = e^{-i\xi_2\widetilde{\Phi}(s,\delta)}\chi_{k}
                    \left(s\right)  \frac{A_k(\xi)}{(1+|\xi|)^{\frac12}}.
\end{equation}

Choose $\beta_0\in C_0^{\infty}(\mathbb{R}^2)$  such that
 \begin{equation*}
 \beta_0(\xi)=\left\{\begin{array}{ll}
 1,& \mbox{if $|\xi|\leq 1$}\,;\\
 0,& \mbox{if $|\xi|\geq 2$}\,.
 \end{array}\right.
 \end{equation*}
For $k\geq1$, define
$ \beta_{k}(\xi)=\beta_{0}(2^{-k}\xi)-\beta_{0}(2^{-(k-1)}\xi). $
Obviously, supp$\beta_{k}\subset\{\xi:2^{k-1}\leq|\xi|\leq 2^{k+1}\}$  for $k\geq1$      and
$ \sum_{k\geq0} \beta_{k}(\xi)=1
$ for any $\xi\in\mathbb{R}^2$.

Let  $\mathcal{A}_j^k=\{A_{j,t}^k\}_{t>0}$, where $$
A_{j,t}^kf(x)= \int_{\mathbb{R}^2} e^{i\xi\cdot x}
  A_{main}^k(t\xi)\beta_{k}(t\xi)\widehat{f}(\xi)d\xi.
$$
Using Lemma \ref{lemma-short-remainder} again, it follows that
$$
\|V^{sh}_q(\mathcal{A}_0^{k}f) \|_{L^p(\mathbb{R}^2)} \leq C_{p} \|f\|_{L^p(\mathbb{R}^2)}
$$
for $q\geq2$ and $p>1$.
It is enough to prove
$$
\sum_{k}2^{-k}\sum_{j\geq1}\left\|V_q^{sh}(\mathcal{A}_{j}^k)\right\|_{L^p\rightarrow L^p}     \leq C_{p,q}
$$
  for $q>\max\{m,p/2\}$ and $p>m$.

\begin{lemma}  \label{lemma-pp}
 \begin{equation} \label{ineq-pp}
 \left\|V_p^{sh}(\mathcal{A}_{j}^kf)\right\|_{L^p(\mathbb{R}^2)}  \leq C_{p,\epsilon_1} 2^{k\cdot\frac{m}{p}}
 2^{-j(\frac{1}{4p}+\epsilon_1)}
\|f\|_{L^p(\mathbb{R}^2)}.
\end{equation}
\end{lemma}
\begin{proof}
By using Lemma \ref{Sobolev embbding}, we see that  $\|V_{p,1}f\|_{L^p}$ is dominated by
\begin{equation} \label{V-1-estimate}
\left\|\int_{\mathbb{R}^2}e^{ix\cdot\xi}
                                        A_{main}^{k}(t\xi)
                                         \beta_j(t\xi)\widehat{f}(\xi)d\xi\right\|^{\frac{1}{p'}}_{L^p}  \times \left\|\int_{\mathbb{R}^2}e^{ix\cdot\xi}\frac{\partial}{\partial t}\left(A^k_{main}(t\xi)\beta_j(t|\xi|)\right)\widehat{f}(\xi)d\xi\right\|^{\frac{1}{p}}_{L^p}, \end{equation}
where we are taking the $L^p$  norm with respect to $\mathbb{R}^2\times[1,2]$.

Moreover,
\begin{eqnarray*} \label{0729003}
&&\frac{\partial}{\partial t}\left(A^k_{main}(t\xi)\beta_j(t|\xi|)\right)
=
e^{-it\xi_2\widetilde{\Phi}(s,\delta)} \chi_{k}\left(s\right)\times   \nonumber \\
&&\left( \left(
-i\xi_2\widetilde{\Phi}(s,\delta)\frac{A_k(t\xi)}{(1+|t\xi|)^{\frac12}}+
\frac{\partial}{\partial t} \left(
\frac{A_k(t\xi)}{(1+|t\xi|)^{\frac12}}\right)\right) \beta_{j}(t|\xi|)
+\frac{A_k(t\xi)}{(1+|t\xi|)^{\frac12}}\frac{\partial}{\partial t} \left(\beta_j(t|\xi|)
\right)
\right).
\end{eqnarray*}
From (\ref{estimate-mainterm}),(\ref{express-mainterm}) and $t|\xi|\approx |\xi| \approx |\xi_2|\approx 2^j$,
 we get
\begin{equation}\label{estimate-082801}
|\xi_2\widetilde{\Phi}(s,\delta)|\lesssim 2^{km}2^{j},\,\,\,\,\,\,\,
\left|\frac{\partial}{\partial t} \left(
\frac{A_k(t\xi)}{(1+|t\xi|)^{\frac12}}\right)\right| \lesssim 2^{-\frac{j}{2}}, \,\,\,\,\,\,\,
\mbox{and\,\,\,\,\,\,\,
$\left|\frac{A_k(t\xi)}{(1+|t\xi|)^{\frac12}}\right|\lesssim 2^{-\frac{j}{2}}$}.
\end{equation}
Thus
the express (\ref{V-1-estimate}) is dominated by  \begin{equation}\label{V-1es}
C 2^{k\cdot \frac{m}{p}}  2^{-j(\frac12-\frac{1}{p})}\left\|\mathcal{F}_{k,j}f\right\|_{L^p(\mathbb{R}^2\times[1,2])},
\end{equation}
where
$$
\mathcal{F}_{k,j}f(x,t)=  \int_{\mathbb{R}^2}e^{ix\cdot \xi}
e^{-it\xi_2\widetilde{\Phi}(s,\delta)}
a_k(\xi,t)   \beta_j(t\xi)
\widehat{f}(\xi)d\xi,
$$
and $a_k(\xi,t)$   is a symbol of order 0 in $\xi$.
More precesely, for arbitrary $t\in[1,2]$,
\begin{equation} \label{0729-1}
|D_{\xi}^{\alpha}a_k(\xi,t)|\leq C_{\alpha}(1+|\xi|)^{-|\alpha|},
\end{equation}
where $C_{\alpha}$   do not depend on $k$  and $t$.

By (\ref{express-Phi}),   it follows that $$
\xi_2\widetilde{\Phi}(s,\delta)=-\frac{m-1}{m}\frac{(-\xi_1)^{\frac{m}{m-1}}}{\xi_2^{\frac{1}{m-1}}}+2^{km}a_0\xi_2+\delta \xi_2R(s,\delta).
$$
The Hessian matrix of the main term $M(\xi):=-\frac{m-1}{m}\frac{(-\xi_1)^{\frac{m}{m-1}}}{\xi_2^{\frac{1}{m-1}}}$ is
$$
-\frac{1}{m-1}\left(\begin{array}{cc}
\frac{(-\xi_1)^{\frac{2-m}{m-1}}}{\xi_2^{\frac{1}{m-1}}} &  \frac{(-\xi_1)^{\frac{1}{m-1}}}{\xi_2^{\frac{m}{m-1}}} \\
\frac{(-\xi_1)^{\frac{1}{m-1}}}{\xi_2^{\frac{m}{m-1}}} &  \frac{(-\xi_1)^{\frac{m}{m-1}}}{\xi_2^{\frac{2m-1}{m-1}}}

\end{array} \right).
$$
Thus the Hessian matrix of $M(\xi)$ has rank one.
                                                         In \cite{1992-MSS-loacal-smoothing} and \cite{1993-MSS-loacal-smoothing},
Mockenhaupt, Seeger and Sogge proved $L^p(\mathbb{R}^2)\rightarrow L^p(\mathbb{R}^2\times[1,2])$
estimates for the operators $\mathcal{F}^*_{k,j}f$,
where
$$
\mathcal{F}^*_{k,j}f(x,t)=  \int_{\mathbb{R}^2}e^{ix\cdot \xi}
e^{-itM(\xi)}
a_k(\xi,t)   \beta_j(t\xi)
\widehat{f}(\xi)d\xi.
$$
They showed that
\begin{equation} \label{F*}
\|\mathcal{F}^*_{k,j}f\|_{L^p(\mathbb{R}^2\times[1,2])} \leq C_{\delta}  2^{j(\frac12-\frac1p-\frac{1}{2p}+\delta(p))}
\|f\|_{L^p(\mathbb{R}^2)}.
\end{equation}

 Furthermore, let $R_k(\xi)$ denote the smooth perturbation term $\xi_2\delta R(s,\delta)$.
Then
\begin{equation}\label{0729-2}
|D^{\alpha}_{\xi}R_k(\xi)|\leq C_{\alpha} (1+|\xi|)^{-|\alpha|},
\end{equation}
where  $C_{\alpha}$   do not depend on $k$.
Using (\ref{0729-1}) and (\ref{0729-2}), the proof of (\ref{F*}) in \cite{1993-MSS-loacal-smoothing}  showed that
the estimate  (\ref{F*}) is valid for oporaters $\mathcal{F}^{**}_{k,j}$ as well,
where
$$
\mathcal{F}^{**}_{k,j}f(x,t)= \int_{\mathbb{R}^2}e^{ix\cdot \xi}
e^{-it\left(M(\xi)+R_k(\xi)\right)}
a_k(\xi,t)   \beta_j(t\xi)
\widehat{f}(\xi)d\xi.
$$

It  is  not hard to check that $\mathcal{F}_{k,j}f(x_1,x_2,t)=\mathcal{F}_{k,j}^{**}f(x_1,x_2+t2^{mk}a_0,t)$.
Since the Lebesgue measure is invariant under translations, it follows that
\begin{equation} \label{F}
\|\mathcal{F}_{k,j}f\|_{L^p(\mathbb{R}^3)} \leq C_{\delta}  2^{j(\frac12-\frac1p-\frac{1}{2p}+\delta(p))}
\|f\|_{L^p(\mathbb{R}^2)}.                                                                                      \end{equation}
 Thus  it follows by (\ref{V-1es}) and (\ref{F}) that
\begin{equation}
\|V_{p,1}(\mathcal{A}_j^kf)\|_{L^p(\mathbb{R}^3)} \leq C_{\delta} 2^{k\cdot \frac{m}{p}} 2^{-j(\frac1{2p}-\delta(p))}
\|f\|_{L^p(\mathbb{R}^2)}.
\end{equation}

Since $V_{p,l}(\mathcal{A}_k^jf)=T_{-k}V_{p,1}(\mathcal{A}_k^j(T_kf))$, we have
 \begin{equation}
 \|V_{p,l}(\mathcal{A}_k^jf)\|_{L^p(\mathbb{R}^3)} \leq C_{\delta} 2^{k\cdot \frac{m}{p}} 2^{-j(\frac1{2p}-\delta(p))}
 \|f\|_{L^p(\mathbb{R}^2)}.
 \end{equation}

Let  $\Delta_l$  be the Littlewood-Paley operator in $\mathbb{R}^2$
defined by $\widehat{\Delta_lf}(\xi)=\widetilde{\beta}(2^{-l}|\xi|)\widehat{f}(\xi)$,
where $\widetilde{\beta}\in C_0^{\infty}(\mathbb{R})$ is nongegative and satisfies
$\beta_0(t|\xi|)=\beta_0(t|\xi|)\widetilde{\beta}(|\xi|)$,
for any $t\in[1,2]$.
Then  by the definition of $V_{p,l}$, we have $$
V_{p,l}(\mathcal{A}_j^kf)(x)=V_{p,l}(\mathcal{A}_j^k(\Delta_{l-j}f))(x). $$
It follows that
\begin{equation}\label{ineq-0830-001}
\left\|V_p^{sh}(\mathcal{A}_j^kf)\right\|_{L^p}\leq C_{\delta} 2^{k\cdot \frac{m}{p}} 2^{-j(\frac1{2p}-\delta(p))} \left\|\left(\sum_l|\Delta_lf|^p\right)^{\frac1p}\right\|_{L^p}
\leq C_{\delta} 2^{k\cdot \frac{m}{p}} 2^{-j(\frac1{2p}-\delta(p))} \|f\|_{L^p}.
\end{equation}
The proof will be completed if we take $\delta(p)=\frac{1}{4p}-\epsilon_1$.
\end{proof}

Next let us prove the following lemma.
\begin{lemma}
\begin{equation}\label{ineq-24082901}
\left\|V_2(\mathcal{A}_{j}^kf)\right\|_{L^4(\mathbb{R}^2)} \leq C_{\epsilon_2}2^{k\cdot\frac{m}{2}} 2^{j\epsilon_2}\|f\|_{L^4(\mathbb{R}^2)}.
\end{equation}
\end{lemma}
\begin{proof}
By Lemma \ref{Sobolev embbding},
\begin{equation*}
 V_{2}^{sh}(\mathcal{A}_j^kf)(x)\leq  C \left(\int_{\mathbb{R}}\left|A_{j,t}^kf(x) \right|^2 \frac{dt}{t} \right)^{\frac{1}{4}}
  \left(\int_{\mathbb{R}}\left|t\frac{\partial}{\partial t}A_{j,t}^kf(x) \right|^2 \frac{dt}{t} \right)^{\frac{1}{4}}. \end{equation*}
Thus, it follows that $\|V_2^{sh}(\mathcal{A}_j^kf)\|_{L^4}$   is bounded by
\begin{eqnarray}\label{equ-240829-11}
&&\left\|\left(\int_{\mathbb{R}}\left|\int_{\mathbb{R}^2}e^{ix\cdot\xi}A_{main}^{k}(t\xi)\beta_j(t\xi)\widehat{f}(\xi)d\xi \right|^2 \frac{dt}{t} \right)^{\frac{1}{2}} \right\|_{L^4(dx)}^{\frac{1}{2}} \times\nonumber\\
&&\left\|\left(\int_{\mathbb{R}}\left|\int_{\mathbb{R}^2}e^{ix\cdot\xi}t\frac{\partial}{\partial t}\left(A^k_{main}(t\xi)\beta_j(t|\xi|)\right)\widehat{f}(\xi)d\xi \right|^2 \frac{dt}{t} \right)^{\frac{1}{2}} \right\|_{L^4(dx)}^{\frac{1}{2}}.
\end{eqnarray}
Using (\ref{estimate-082801}), we conclude $\|V_2^{sh}(\mathcal{A}_j^kf)\|_{L^4}$ is dominated by
\begin{equation} \label{estimate-squrefuncion}
2^{k\cdot\frac{m}{2}} \left\|\left(\int_{\mathbb{R}}|\mathcal{F}_{k,j}f(\cdot,t)|^2\frac{dt}{t}\right)^{\frac12}\right\|_{L^4}. \end{equation}
%where
%$$
% \mathcal{P}_{k,j}f(x,t)= \int_{\mathbb{R}^2}e^{ix\cdot\xi}e^{it\xi_2\widetilde{\Phi}(s,\delta)}a_k(t\xi)\beta_j(t\xi)\widehat{f}(\xi)d\xi
%$$
%with $a_k(\xi)$ satisfying,
%\begin{equation}
%|D_{\xi}^{\alpha}a_k(\xi)| \leq C_{\alpha}  (1+|\xi|)^{-|\alpha|}, \end{equation}
%where $C_{\alpha}$ do not depend on $k$.

Using Theorem 6.2 in \cite{1993-MSS-loacal-smoothing}, Lemma \ref{multiplier-Th} and the same method as in the proof of (\ref{F}) , we have
$$
\left\|\left(\int_{\mathbb{R}}\left|\mathcal{F}_{k,j}f(\cdot,t)\right|^2\frac{dt}{t}\right)^{\frac12} \right\|_{L^4(\mathbb{R}^2)}
\leq  C_{\epsilon_2}2^{j\epsilon_2} \|f\|_{L^4(\mathbb{R}^2)}.
$$
By (\ref{estimate-squrefuncion}), we obtain
\begin{equation*}   \label{inequ-24}
\|V_2^{sh}(\mathcal{A}_j^kf)\|_{L^4}\leq C_{\epsilon_2} 2^{k\cdot\frac{m}{2}} 2^{j\epsilon_2} \|f\|_{L^4}.
\end{equation*}
\end{proof}

Using (\ref{ineq-pp})  and (\ref{ineq-24082901}),
interpolation between $\|V_{\infty}^{sh}(\mathcal{A}_{j}^kf)\|_{L^\infty}$  and $\|V_{2}^{sh}(\mathcal{A}_{j}^kf)\|_{L^4}$ implies
\begin{equation} \label{inequ-p2p-1}
\left\|V_{\frac{p}{2}}^{sh}(\mathcal{A}_j^kf)  \right\|_{L^p}\leq  C_{\epsilon_1,\epsilon_2} 2^{k\cdot\frac{2m}{p}} 2^{-j\left(\epsilon_2\frac{4}{p}
-\epsilon_1(1-\frac{4}{p})\right)}\|f\|_{L^p},
\end{equation}
for $p\geq 4$.
Thus if we  take $\epsilon_1=\epsilon$, $\epsilon_2=\epsilon(\frac{p}{4}-1)$, then
\begin{equation} \label{inequ-p2p}
\left\|V_q^{sh}(\mathcal{A}_j^kf)  \right\|_{L^p}\leq  C_{\epsilon} 2^{k\cdot\frac{2m}{p}} 2^{-j\epsilon(\frac12-\frac{2}{p})}\|f\|_{L^p}
\end{equation}
holds for $q>p/2$.
It follows  for $p>2m$ and $q>p/2$ $$\sum_k2^{-k}\sum_{j\geq1} \left\|V_q^{sh}(\mathcal{A}_j^k)  \right\|_{L^p\rightarrow L^p} \leq C. $$

Using (\ref{ineq-pp})  and (\ref{inequ-p2p-1}),
interpolation between $\|V_{p}^{sh}(\mathcal{A}_{j}^kf)\|_{L^p}$  and $\|V_{\frac{p}{2}}^{sh}(\mathcal{A}_{j}^kf)\|_{L^p}$ implies,  for $p/2\leq q\leq p$ ,
\begin{equation} \label{inequ-qp-1}
\left\|V_q^{sh}(\mathcal{A}_j^kf)  \right\|_{L^p}\leq  C_{\tilde{\epsilon}} 2^{k\cdot\frac{m}{q}} 2^{-j\tilde{\epsilon}}\|f\|_{L^p},
\end{equation}
where $\tilde{\epsilon}=\left(\frac{1}{4p}+\epsilon_2\right)(2-\frac{p}{q})+\left(\frac{4}{p}\epsilon_2-\left(1-\frac{4}{p}\right)\epsilon_1\right)\left(\frac{p}{q}-1\right)$.
Thus, taking  $\epsilon_1=\epsilon$, $\epsilon_2=\epsilon(\frac{p}{4}-1)$,  we have for $m<p\leq 2m$ and $q>m$
$$\sum_k2^{-k}\sum_{j\geq1} \left\|V_q^{sh}(\mathcal{A}_j^k)  \right\|_{L^p\rightarrow L^p} \leq C. $$

\subsection{The proof of the sufficient condition for the case when $a_0= 0$}
The proof for this case is essentially the same as Case $a_0\neq0$, with the only distinction being in the process of demonstration.

\textit{ The estimation of  $\|V_{p}^{sh}(\mathcal{A}_j^kf)\|_{L^p}$.}
From (\ref{estimate-mainterm}), (\ref{express-mainterm}) and $t|\xi|\approx |\xi| \approx |\xi_2|\approx 2^j$, because $a_0=0$,
 we get
\begin{equation}\label{estimate-00829-1}
|\xi_2\widetilde{\Phi}(s,\delta)|\lesssim 2^{j},\,\,\,\,\,\,\,
\left|\frac{\partial}{\partial t} \left(
\frac{A_k(t\xi)}{(1+|t\xi|)^{\frac12}}\right)\right| \lesssim 2^{-\frac{j}{2}}, \,\,\,\,\,\,\,
\mbox{and
$\left|\frac{A_k(t\xi)}{(1+|t\xi|)^{\frac12}}\right|\lesssim 2^{-\frac{j}{2}}$}.
\end{equation}
Thus
the express (\ref{V-1-estimate}) is dominated by  \begin{equation}\label{V-1es-0829}
C   2^{-j(\frac12-\frac{1}{p})}\left\|\mathcal{F}_{k,j}f\right\|_{L^p(\mathbb{R}^2\times[1,2])},
\end{equation}
where
$$
\mathcal{F}_{k,j}f(x,t)=  \int_{\mathbb{R}^2}e^{ix\cdot \xi}
e^{-it\xi_2\widetilde{\Phi}(s,\delta)}
a_k(\xi,t)   \beta_j(t\xi)
\widehat{f}(\xi)d\xi,
$$
and $a_k(\xi,t)$   is a symbol of order 0 in $\xi$.

By (\ref{F*}), we obtain
$$\left\|\mathcal{F}_{k,j}f\right\|_{L^p(\mathbb{R}^2\times[1,2])}\leq C_{\delta}  2^{j(\frac12-\frac1p-\frac{1}{2p}+\delta(p))}\|f\|_{L^p(\mathbb{R}^2)}. $$
It follows that
 \begin{equation}\label{V-1es-0829-1}
\|V_{p,1}(\mathcal{A}_j^kf)\|_{L^p}\leq C   2^{-j(\frac{1}{2p}-\delta(p))}\|f\|_{L^p(\mathbb{R}^2)}.
\end{equation}

Using the Littlewood-Paley theorem again and taking $\delta(p)=\frac{1}{4p}-\epsilon_1$, we have
\begin{equation}\label{V-1es-0829-2}
\|V_{p}^{sh}(\mathcal{A}_j^kf)\|_{L^p}\leq C_{\epsilon}   2^{-j(\frac{1}{4p}+\epsilon_1)}\|f\|_{L^p(\mathbb{R}^2)}.
\end{equation}

\textit{ The estimation of  $\|V_{2}^{sh}(\mathcal{A}_j^kf)\|_{L^4}$.}
By using (\ref{estimate-00829-1}) and (\ref{equ-240829-11}), we have

$$
\|V_{2}^{sh}(\mathcal{A}_j^kf)\|_{L^4}\leq  \left\|\left(\int_{\mathbb{R}}|\mathcal{F}_{k,j}f(\cdot,t)|^2\frac{dt}{t}\right)^{\frac12}\right\|_{L^4}.
$$
Thus  Theorem 6.2 in \cite{1993-MSS-loacal-smoothing} and Lemma \ref{multiplier-Th} ensure
$$
\left\|\left(\int_{\mathbb{R}}\left|\mathcal{F}_{k,j}f(\cdot,t)\right|^2\frac{dt}{t}\right)^{\frac12} \right\|_{L^4(\mathbb{R}^2)}
\leq  C_{\epsilon_2}2^{j\epsilon_2} \|f\|_{L^4(\mathbb{R}^2)}.
$$
It can be deduced that
\begin{equation}   \label{inequ-24-90-29}
\|V_2^{sh}(\mathcal{A}_j^kf)\|_{L^4}\leq C_{\epsilon_2}  2^{j\epsilon_2} \|f\|_{L^4}.
\end{equation}

Using (\ref{V-1es-0829-2})  and (\ref{inequ-24-90-29}),
interpolation between $\|V_{\infty}^{sh}(\mathcal{A}_{j}^kf)\|_{L^\infty}$  and $\|V_{2}^{sh}(\mathcal{A}_{j}^kf)\|_{L^4}$ implies
\begin{equation} \label{inequ-p2p-1-0829}
\left\|V_{\frac{p}{2}}^{sh}(\mathcal{A}_j^kf)  \right\|_{L^p}\leq  C_{\epsilon_1,\epsilon_2}  2^{-j\left(\epsilon_2\frac{4}{p}
-\epsilon_1(1-\frac{4}{p})\right)}\|f\|_{L^p},
\end{equation}
for $p\geq 4$.
Thus if we  take $\epsilon_1=\epsilon$, $\epsilon_2=\epsilon(\frac{p}{4}-1)$, then
\begin{equation} \label{inequ-p2p-0829}
\left\|V_q^{sh}(\mathcal{A}_j^kf)  \right\|_{L^p}\leq  C_{\epsilon}  2^{-j\epsilon(\frac12-\frac{2}{p})}\|f\|_{L^p}
\end{equation}
holds for $q>p/2$.
It follows  for $p>4$ and $q>p/2$ $$\sum_k2^{-k}\sum_{j\geq1} \left\|V_q^{sh}(\mathcal{A}_j^k)  \right\|_{L^p\rightarrow L^p} \leq C. $$

Using (\ref{V-1es-0829-2})  and (\ref{inequ-p2p-1-0829}),
interpolation between $\|V_{p}^{sh}(\mathcal{A}_{j}^kf)\|_{L^p}$  and $\|V_{\frac{p}{2}}^{sh}(\mathcal{A}_{j}^kf)\|_{L^p}$ implies,  for $p/2\leq q\leq p$ ,
\begin{equation} \label{inequ-qp-1-0829}
\left\|V_q^{sh}(\mathcal{A}_j^kf)  \right\|_{L^p}\leq  C_{\epsilon}  2^{-j\epsilon}\|f\|_{L^p},
\end{equation}
where $\epsilon=\left(\frac{1}{4p}+\epsilon_2\right)(2-\frac{p}{q})+\left(\frac{4}{p}\epsilon_2-\left(1-\frac{4}{p}\right)\epsilon_1\right)\left(\frac{p}{q}-1\right)$.
Thus, taking  $\epsilon_1=\epsilon$, $\epsilon_2=\epsilon(\frac{p}{4}-1)$,  we have for $1<p\leq 4$ and $q>2$
$$\sum_k2^{-k}\sum_{j\geq1} \left\|V_q^{sh}(\mathcal{A}_j^k)  \right\|_{L^p\rightarrow L^p} \leq C. $$

\subsection{The proof of the necessary condition}
We now turn to prove the necessary condition of the statement in  Theorem \ref{Th-isotropic}  and show that the a priori
whenever $a_0=0$ or $a_0\neq 0$,
inequality
$$
\|V_q(\mathcal{A}f)\|_{L^p}\leq C_p\|f\|_{L^p}
$$
 implies that $q\geq p/2$.
  The proof of the case $a_0\neq0$ is presented here, as the proof for the case $a_0=0$ follows a similar approach.
 For the sake of simplicity, let us assume that $a_0=1$.

\subsubsection{The case for $\phi(y)=1-\frac{1}{m}y^m$}.

\textit{Preparatory work. }
We choose $\phi(y)=1-\frac{1}{m}y^m$ and denote
$
\widehat{d\mu}(\xi)=\int e^{-i\left(\xi_1y+\xi_2(1-\frac{1}{m}y^m)\right)}\eta(y)dy.
$
Theorem \ref{stationary phase} ensures
$$
\widehat{d\mu}(\xi)=\exp\left(i\left(\frac{m-1}{m}\cdot\frac{\xi_1^{\frac{m}{m-1}}}{\xi_2^{\frac{1}{m-1}}}+\xi_2\right)\right)
\chi\left(\frac{\xi_1}{\xi_2}\right)a(\xi)+B(\xi),
$$
where $\chi$ is a smooth function supported in a small neighbourhood near zero
and
\begin{equation}\label{2024-inq-D}
|D_\xi^{\alpha} a(\xi)|\leq C_{\alpha}(1+|\xi|)^{-1/2-|\alpha|}
\end{equation}
and
\begin{equation}\label{2024-inq-DB}
|D_\xi^{\alpha} B(\xi)|\leq C_{\alpha,N}(1+|\xi|)^{-N}.
\end{equation}

Denote
$$
A_{1,t}f(x)=\frac{1}{(2\pi)^2}\int_{\mathbb{R}^2}e^{ix\cdot\xi}
\exp\left(it\left(\frac{m-1}{m}\cdot\frac{\xi_1^{\frac{m}{m-1}}}{\xi_2^{\frac{1}{m-1}}}+\xi_2\right)\right)
\chi\left(\frac{\xi_1}{\xi_2}\right)a(t\xi)\widehat{f}(\xi)d\xi
           $$
and
$$
A_{2,t}f(x)=\frac{1}{(2\pi)^2}\int_{\mathbb{R}^2}e^{ix\cdot\xi}B(t\xi)\widehat{f}(\xi)d\xi.
$$
It follows that
$
A_tf(x)=A_{1,t}f(x)+A_{2,t}f(x).
$

Let $\eta_1=\frac{m-1}{m}\cdot\frac{\xi_1^{\frac{m}{m-1}}}{\xi_2^{\frac{1}{m-1}}}+\xi_2$
and $\eta_2=\frac{\xi_1}{\xi_2}$.
Then the express of the main term $A_{1,t}f$ is
%\begin{equation*}\label{general-gn2}
%\left\{\begin{array}{l}
%\eta_1=\frac{m-1}{m}\cdot\frac{\xi_1^{\frac{m}{m-1}}}{\xi_2^{\frac{1}{m-1}}}+\xi_2\,,\\
%\eta_2=\frac{\xi_1}{\xi_2}\,.
%\end{array}\right.
%\end{equation*}
%Then
%\begin{equation*}\label{general-gn22}
%\left\{\begin{array}{l}
%\xi_1=\frac{\eta_2}{1+\frac{m-1}{m}\eta_2^{\frac{m}{m-1}}}\eta_1\,,\\
%\xi_2=\frac{1}{1+\frac{m-1}{m}\eta_2^{\frac{m}{m-1}}}\eta_1\,.
%\end{array}\right.
%\end{equation*}

\begin{eqnarray}\label{240829-express-mainterm}
&&\frac{1}{(2\pi)^2}\int_{\mathbb{R}^2}\exp\left( i\left(\frac{\eta_2}{1+\frac{m-1}{m}\eta_2^{\frac{m}{m-1}}}\eta_1x_1+
\frac{1}{1+\frac{m-1}{m}\eta_2^{\frac{m}{m-1}}}\eta_1x_2\right)\right)\times\nonumber\\
&&\,\,\,\,\,\,\,\,\,\,\,\,\,\,\,\,\,\,\,\,\,\,\,\,\,\,\,\,\,\,\,\,e^{it\eta_1}\chi(\eta_2)\widetilde{a}(t\eta_1,\eta_2)\widehat{\tilde{f}}(\eta)
\left|
\frac{\eta_1}{\left(1+\frac{m-1}{m}\eta_2^{\frac{m}{m-1}}\right)^2}
\right|
d\eta
\end{eqnarray}
and the express of the remainder term $A_{2,t}f$ is
\begin{eqnarray}\label{240829-estimate-remainterm}
&&\frac{1}{(2\pi)^2}\int_{\mathbb{R}^2}\exp\left( i\left(\frac{\eta_2}{1+\frac{m-1}{m}\eta_2^{\frac{m}{m-1}}}\eta_1x_1+
\frac{1}{1+\frac{m-1}{m}\eta_2^{\frac{m}{m-1}}}\eta_1x_2\right)\right)\times\\
&&\,\,\,\,\,\,\,\,\,\,\,\,\,\,\,\,\,\,\,\,\,\,\,\,\,\,\,\,\,\,\,\,\tilde{B}(t\eta_1,\eta_2)\widehat{\tilde{f}}(\eta)
\left|
\frac{\eta_1}{\left(1+\frac{m-1}{m}\eta_2^{\frac{m}{m-1}}\right)^2}
\right|d\eta,
\end{eqnarray}
where
$$\widehat{\tilde{f}}(\eta)=\widehat{f}\left(\frac{\eta_2}{1+\frac{m-1}{m}\eta_2^{\frac{m}{m-1}}}\eta_1,\frac{1}{1+\frac{m-1}{m}\eta_2^{\frac{m}{m-1}}}\eta_1\right),$$
$$
\tilde{a}(t\eta_1,\eta_2)=a\left(\frac{\eta_2}{1+\frac{m-1}{m}\eta_2^{\frac{m}{m-1}}}(t\eta_1),
\frac{1}{1+\frac{m-1}{m}\eta_2^{\frac{m}{m-1}}}(t\eta_1)\right)
$$
and
$$
\tilde{B}(t\eta_1,\eta_2)=B\left(\frac{\eta_2}{1+\frac{m-1}{m}\eta_2^{\frac{m}{m-1}}}(t\eta_1),
\frac{1}{1+\frac{m-1}{m}\eta_2^{\frac{m}{m-1}}}(t\eta_1)\right).
$$
%Next let us estimate $\tilde{a}(t\eta_1,\eta_2)$ and $\tilde{B}(t\eta_1,\eta_2)$.
%\begin{eqnarray*}
%\frac{\partial}{\partial \eta_1}\left(\tilde{a}(t\eta_1,\eta_2)\right)=
%(\partial_1a)(t\xi_1,t\xi_2)\cdot \frac{t\eta_2}{1+\frac{m-1}{m}\eta_2^{\frac{m}{m-1}}}
%+(\partial_2a)(t\xi_1,t\xi_2)\cdot \frac{t}{1+\frac{m-1}{m}\eta_2^{\frac{m}{m-1}}}
%\end{eqnarray*}
%\begin{eqnarray*}
%\frac{\partial}{\partial \eta_2}\left(\tilde{a}(t\eta_1,\eta_2)\right)=
%(\partial_1a)(t\xi_1,t\xi_2)\cdot \frac{1-\frac{1}{m}\eta_2^{\frac{m}{m-1}}}{\left(1+\frac{m-1}{m}\eta_2^{\frac{m}{m-1}}\right)^2}\cdot(t\eta_1)
%+(\partial_2a)(t\xi_1,t\xi_2)\cdot \frac{-\eta_2^{\frac{1}{m-1}}}{\left(1+\frac{m-1}{m}\eta_2^{\frac{m}{m-1}}\right)^2}\cdot(t\eta_1)
%\end{eqnarray*}
Since $\eta_2$ can be assume very small, we have
%$$
%\left|\frac{\partial}{\partial \eta_1}\left(\tilde{a}(t\eta_1,\eta_2)\right)\right|\leq Ct(1+|t\xi|)^{-\frac{3}{2}}\leq
%Ct(1+t|\eta_1|)^{-\frac{3}{2}}
%$$
%and
%$$
%\left|\frac{\partial}{\partial \eta_2}\left(\tilde{a}(t\eta_1,\eta_2)\right)\right|\leq Ct|\eta_1|(1+|t\xi|)^{-\frac{3}{2}}\leq
%C(1+t|\eta_1|)^{-\frac{1}{2}}.
%$$
%It follows that
\begin{equation}\label{240829-estimate-main-term}
\left|\frac{\partial^{l+k}}{(\partial \eta_1)^l(\partial \eta_2)^k}\left(\tilde{a}(t\eta_1,\eta_2)\right)\right|\leq C_{l,k} t^l(1+t|\eta_1|)^{-\frac{1}{2}-l},
\end{equation}
and
\begin{equation}\label{240829-estimate-remain-term}
\left|\frac{\partial^{l+k}}{(\partial \eta_1)^l(\partial \eta_2)^k}\left(B(t\eta_1,\eta_2)\right)\right|\leq C_{l,k,N} t^l(1+t|\eta_1|)^{-N}.
\end{equation}

\textit{Estimation of the upper bound for $\|f\|_{L^p}$.}
Choose $\widehat{\tilde{f}}_{\lambda}(\eta)=\chi_0\left(\frac{\eta_1}{\lambda}\right)\chi_1(\eta_2)
e^{-i\frac{\eta_1^2}{2\lambda}}$,  where $\chi_0\left(\eta_1\right)$ and $\chi_1\left(\eta_1\right)$ is supported in a very small interval $[B/2,2B]$ near zero.
Now we try to estimate $\|f_{\lambda}\|_{L^p}$.

After a change of variable, we have
\begin{equation}\label{equ-0915-03}
f_{\lambda}(x)=\frac{1}{2\pi}\int_{\mathbb{R}}e^{-i\frac{\eta_1^2}{2\lambda}}h(x,\eta_1)\chi_0\left(\frac{\eta_1}{\lambda}\right)\eta_1d\eta_1,
\end{equation}
where
$$
h(x,\eta_1)=\frac{1}{2\pi}\int_{\mathbb{R}}\exp\left(i\eta_1\cdot\frac{\eta_2x_1+x_2}{1+\frac{m-1}{m}\eta_2^{\frac{m}{m-1}}}\right)
\chi_1(\eta_2)\frac{1}{\left(1+\frac{m-1}{m}\eta_2^{\frac{m}{m-1}}\right)^2}d\eta_2.
$$

In order to use stationary phase theorem, let us compute
$$
\frac{\partial}{\partial\eta_2}\left(\frac{\eta_2x_1+x_2}{1+\frac{m-1}{m}\eta_2^{\frac{m}{m-1}}}\right)=
\frac{x_1\left(1-\frac{1}{m}\eta_2^{\frac{m}{m-1}}\right)-x_2\eta_2^{\frac{1}{m-1}}}{\left(\frac{\eta_2x_1+x_2}{1+\frac{m-1}{m}\eta_2^{\frac{m}{m-1}}}\right)^2}.
$$
Denote  the numerator by $L(\eta_2,s):=x_2\left(s\left(1-\frac{1}{m}\eta_2^{\frac{m}{m-1}}\right)-\eta_2^{\frac{1}{m-1}}\right)$
where $s=\frac{x_1}{x_2}$.
Since $L(0,s)=x_2s$ and $L(1,s)=x_2(\frac{m-1}{m}s-1)$,
it follows that there exists a number $\tilde{\eta}_2=\tilde{\eta}_2(s)$ satisfying
 $\tilde{\eta}_2(0)=0$
 and
$$
\frac{\partial}{\partial\eta_2}\left(\frac{\eta_2 x_1+x_2}{1+\frac{m-1}{m}\eta_2^{\frac{m}{m-1}}}\right)=0
$$
for $0\leq s<\frac{m-1}{m}$.

Using  the method of
stationary phase, we have
\begin{equation}\label{equ-091502}
h(x,\eta_1)=c_0\exp\left(i\eta_1\cdot \left(\frac{\tilde{\eta_2}(s) x_1+x_2}{1+\frac{m-1}{m}\tilde{\eta_2}(s)^{\frac{m}{m-1}}}\right)\right)
\cdot\frac{N(\eta_1,x)}{(1+\eta_1)^{\frac12}}\cdot \tilde{\chi}(s),
\end{equation}
where $N(x,\eta_1)$ is a symbol of order zero in $\eta_1$ and $\tilde{\chi}(s)$ is supported in an interval near zero.

Employing stretching transformation for (\ref{equ-0915-03}), we have
$$
f_{\lambda}(x)=\frac{\lambda^2}{2\pi}\int_{\mathbb{R}}e^{-i\frac{\lambda\eta_1^2}{2}}h(x,\lambda\eta_1)\chi_0\left(\eta_1\right)\eta_1d\eta_1.
$$
Using the method of stationary phase,
it follows that $\|f_{\lambda}\|_{L^p}
=\mathcal{O}(\lambda)$.

\textit{Estimation of the lower bound for $\|V_q(\mathcal{A}f)\|_{L^p}$.}
In order to derive the lower bound for $V_q(\mathcal{A}f)(x)$, we restrict $x$  satisfying $0<x_2\leq c_0\lambda^{-1}$ and $c_1\leq \frac{x_1}{x_2}\leq \tilde{c_1}$.

It can be deduced from (\ref{240829-estimate-remain-term}) that
$A_{2,t}f_{\lambda}(x)=\mathcal{O}(\lambda^{-N})$.
Substituting $\tilde{f}_{\lambda}$ into the expression (\ref{240829-express-mainterm})  yields,
\begin{eqnarray}\label{240830-estimate-001}
A_{1,t}f_{\lambda}(x)&=&\frac{1}{2\pi}\int_{\mathbb{R}}e^{it\eta_1}\chi_0\left(\frac{\eta_1}{\lambda}\right)e^{-i\frac{\eta_1^2}{2\lambda}}G(\eta_1,x,t)\eta_1d\eta_1\nonumber\\
&=&\frac{\lambda^2}{2\pi}e^{i\cdot\frac{\lambda t^2}{2}}\int_{\mathbb{R}}e^{-i\cdot\frac{\lambda}{2}(\eta_1-t)^2}
\chi_0(\eta_1)G(\lambda\eta_1,x,t)\eta_1d\eta_1 ,
\end{eqnarray}
where the express of $G(\lambda\eta_1,x,t)$ is
$$
\frac{1}{2\pi}\int_{\mathbb{R}}\exp\left( i(\lambda x_2)\eta_1\left(\frac{\eta_2 \cdot \frac{x_1}{x_2}}{1+\frac{m-1}{m}\eta_2^{\frac{m}{m-1}}}\right)\right)
\chi(\eta_2)\chi_1(\eta_2)\tilde{a}(t\eta_1,\eta_2)\frac{1}{\left(1+\frac{m-1}{m}\eta_2^{\frac{m}{m-1}}\right)^2}d\eta_2.
$$
By $0<\lambda x_2\leq c_0$,
$c_1\leq \frac{x_1}{x_2}\leq \tilde{c_1}$, $t\approx 1$
and (\ref{240829-estimate-main-term}), we could obtain
\begin{equation}\label{equ-0915-01}
G(\lambda\eta_1,x,t)=\frac{M(\lambda\eta_1,x,t)}{(1+\lambda\eta_1)^{\frac12}}.
\end{equation}
Employing   stationary phase theorem, (\ref{equ-0915-01}) and (\ref{240830-estimate-001}), we have
\begin{equation*}
A_{1,t}f_{\lambda}(x)
=c_0\lambda e^{i\cdot\frac{\lambda t^2}{2}}
 t\chi_0(t) M(t\lambda,x,t)
 +\mathcal{O}(1).
\end{equation*}

Denote $$I(x,t)=c_0\lambda e^{i\cdot\frac{\lambda t^2}{2}}
 t\chi_0(t) M(t\lambda,x,t).$$
For $1\leq n\leq \lambda/100$, we choose $$t_n=\left(2+\frac{2n\pi}{\lambda}\right)^{\frac12}.$$
Then $|t_n-t_{n+1}|=\mathcal{O}(\lambda^{-1})$.

By $0<x_2\leq c_0\lambda^{-1}$,
we have $
|I(x,t_n)-I(x,t_{n+1})| \gtrsim \lambda(c-C\lambda^{-1})
$.
Therefor,
$$
\left(\sum_{1\leq n<\frac{\lambda}{100}}|I(x,t_n)-I(x,t_{n+1})|^q\right)^{\frac1q}\gtrsim \lambda(c-C\lambda^{-1}).
$$
 Combining the terms above yields for $0<x_2<c_0\lambda^{-1}$ and $c_1<\frac{x_1}{x_2}<\tilde{c_1}$, we have
 $$
 V_q(\mathcal{A}f_{\lambda})(x)\gtrsim  \lambda(c-C\lambda^{-1})
 $$
and consequently $\|V_q(\mathcal{A}f_{\lambda})\|_{L^p}\gtrsim \lambda^{1+\frac1q-\frac2p}$.
Since $\|f_{\lambda}\|_{L^p}\leq C\lambda$,  this yields the restriction $q\geq p/2$.
\subsubsection{General case}
Suppose  $\phi(y)=1-\frac{1}{m}y^mb(y)$ and $b(0)=1$.
It follows that
$$
\widehat{d\mu}(\xi)=\int e^{-i\left(\xi_1y+\xi_2(1-\frac{1}{m}y^mb(y))\right)}\eta(y)dy.
$$
Denote the phase of $\widehat{d\mu}(\xi)$ by $$\Phi(y,s)=ys+1-\frac{1}{m}y^mb(y),$$ where $s=\frac{\xi_1}{\xi_2}$.
Note that $$\Phi(y,s)=ys+1-\frac{1}{m}y^m+\mathcal{O}(y^{m+1})$$
and
\begin{equation}\label{2024080801}
\partial_1\Phi(y,s)=s-y^{m-1}+\mathcal{O}(y^m).
\end{equation}
It follows that the resolution of above function can be written as
$$
 \tilde{y}=s^{\frac{1}{m-1}}+\mathcal{O}(s^{\frac{m}{m-1}}).
$$

By the method of stationary phase, we have
$$
\widehat{d\mu}(\xi)=\exp\left(i\left(\frac{m-1}{m}\cdot\frac{\xi_1^{\frac{m}{m-1}}}{\xi_2^{\frac{1}{m-1}}}+\xi_2+\xi_2\mathcal{O}(s^{\frac{m}{m-1}})\right)\right)
\chi\left(\frac{\xi_1}{\xi_2}\right)a(\xi)+B(\xi),
$$
where $\chi$ is a smooth function supported in a small neighbourhood near zero
and
$$|D_\xi^{\alpha} a(\xi)|\leq C_{\alpha}(1+|\xi|)^{-1/2-|\alpha|}$$
and
$$|D_\xi^{\alpha} B(\xi)|\leq C_{\alpha,N}(1+|\xi|)^{-N}.$$

It follows that the express of the  main term $A_{1,t}f$ is changed to
%$$
%A_tf(x)=A_{1,t}f(x)+A_{2,t}f(x),
%$$
%where
$$
\frac{1}{(2\pi)^2}\int_{\mathbb{R}^2}e^{ix\cdot\xi}
\exp\left(it\left(\frac{m-1}{m}\cdot\frac{\xi_1^{\frac{m}{m-1}}}{\xi_2^{\frac{1}{m-1}}}+\xi_2+\xi_2\mathcal{O}(s^{\frac{m}{m-1}})\right)\right)
\chi\left(\frac{\xi_1}{\xi_2}\right)a(t\xi)\widehat{f}(\xi)d\xi.
           $$
%and
%$$
%A_{2,t}f(x)=\frac{1}{(2\pi)^2}\int_{\mathbb{R}^2}e^{ix\cdot\xi}B(t\xi)\widehat{f}(\xi)d\xi.
%$$

Let $\eta_1=\frac{m-1}{m}\cdot\frac{\xi_1^{\frac{m}{m-1}}}{\xi_2^{\frac{1}{m-1}}}+\xi_2$
and $\eta_2=\frac{\xi_1}{\xi_2}$ as well.
%\begin{equation*}\label{general-gn2}
%\left\{\begin{array}{l}
%\eta_1=\frac{m-1}{m}\cdot\frac{\xi_1^{\frac{m}{m-1}}}{\xi_2^{\frac{1}{m-1}}}+\xi_2\,,\\
%\eta_2=\frac{\xi_1}{\xi_2}\,.
%\end{array}\right.
%\end{equation*}
%Then
%\begin{equation*}\label{general-gn22}
%\left\{\begin{array}{l}
%\xi_1=\frac{\eta_2}{1+\frac{m-1}{m}\eta_2^{\frac{m}{m-1}}}\eta_1\,,\\
%\xi_2=\frac{1}{1+\frac{m-1}{m}\eta_2^{\frac{m}{m-1}}}\eta_1\,.
%\end{array}\right.
%\end{equation*}
%
%It follows that
Then  the express of the  main term $A_{1,t}f$ become
\begin{eqnarray}\label{240830-mainterm}
&&\frac{1}{(2\pi)^2}\int_{\mathbb{R}^2}\exp\left( i\left(\frac{\eta_2}{1+\frac{m-1}{m}\eta_2^{\frac{m}{m-1}}}\eta_1x_1+
\frac{1}{1+\frac{m-1}{m}\eta_2^{\frac{m}{m-1}}}\eta_1x_2\right)\right)\times\nonumber\\
&&\exp\left({it(\eta_1+\frac{1}{1+\frac{m-1}{m}\eta_2^{\frac{m}{m-1}}}\eta_1\mathcal{O}(\eta_2^\frac{m}{m-1}))}\right)\chi(\eta_2)\tilde{a}(t\eta_1,\eta_2)\widehat{\tilde{f}}(\eta)
\left|
\frac{\eta_1}{\left(1+\frac{m-1}{m}\eta_2^{\frac{m}{m-1}}\right)^2}
\right|
d\eta.
\end{eqnarray}
Comparing the express above and (\ref{240829-express-mainterm}),
the disturbance term $\mathcal{O}(\eta_2^\frac{m}{m-1})$ could be ignored,
since $\eta_2$ is a very small number.
Thus we  complete the proof.

\section{The proof for variational operators associated with non-isotropic dilations of curves in the plane}
\subsection{The proof of the sufficient condition}
We proceed as the proof of Theorem \ref{Th-isotropic}.
Denote
$$\widehat{d\mu}(\xi)=\int_{\mathbb{R}}e^{-i(y\xi_1+y^mb(y)\xi_2)}\eta(y)dy.$$
By van der corput's lemma, we have $|\widehat{d\mu}(\xi)|\leq C(1+|\xi|)^{-1/m}$.
It ensures  us only to consider the short-variation   $V_q^{sh}(\mathcal{B}f)$.

We choose $B>0$  enough small so that the interval $(-2B, 2B)$ does not contain any other flat point of the curve $(y,\phi(y))$ beside  $(0,0)$.
Define $\rho_0\in C_0^{\infty}(\mathbb{R})$ as before.
%Thus,
%\begin{eqnarray*}
%B_t(f)(x)&=&\sum_{k}\int_{\mathbb{R}} f(x_1-ty,x_2-t^my^mb(y))\eta(y) \rho_0(2^ky)dy\\
%&=& \sum_{k}2^{-k}\int_{\mathbb{R}} f(x_1-t2^{-k}y,x_2-(2^{-k}t)^my^mb(2^{-k}y)))\eta(2^{-k}y) \rho_0(y)dy \\
%&=& \sum_{k} 2^{-k} T_k^{-1}B_t^kT_kf(x) ,
%\end{eqnarray*}
Put  $$B_t^kf(x):=\int f(x_1-ty,x_2-t^my^mb(2^{-k}y))\rho_0(y)\eta(2^{-k}y)dy.$$
As in the argument in the proof of Theorem \ref{Th-isotropic},
 it suffices to prove
$$
\sum_{k}2^{-k}\left\|V_q^{sh}(\mathcal{B}^k)\right\|_{L^p\rightarrow L^p}     \leq C_{p,q}
$$
for $q>\max\{2,p/2\}$ and $p>2$.

Using the Fourier inversion formula, we have
\begin{eqnarray*}
B_t^k(f)(x)&=&\frac{1}{(2\pi)^2}\int_{\mathbb{R}^2}e^{ix\cdot\xi}
\int_{\mathbb{R}} e^{-i
\left(ty\xi_1+t^my^mb(2^{-k}y)\xi_2\right)}\eta(2^{-k}y)\rho_0(y)dy\widehat{f}(\xi)d\xi\\
&=& \frac{1}{(2\pi)^2}\int_{\mathbb{R}^2}e^{ix\cdot\xi}\widehat{d\mu_k}(\delta_t\xi)\widehat{f}(\xi)d\xi,
\end{eqnarray*}
where
$$
\widehat{d\mu_k}(\delta_t\xi):=\int_{\mathbb{R}}
e^{-it^m\xi_2\Phi(y,s,\delta)}\eta( \delta y)\rho_0(y)dy,
$$
with $s=-\frac{\xi_1}{t^{m-1}\xi_2} $, $\delta=2^{-k}$, and $\Phi(y,s,\delta)=-ys+y^{m}b(\delta y)$.

Then we have
\begin{equation}  \label{001--2}
\left(\frac{\partial }{\partial y}\right)^2\Phi(0,s,\delta) = 0, \,\,\, \cdots,\,\,\,\left(\frac{\partial }{\partial y}\right)^{m-1}\Phi(0,s,\delta) = 0,\,\,\,
\left(\frac{\partial }{\partial y}\right)^m\Phi(0,s,\delta)=m!b(0)  \neq 0.
\end{equation}
Note that we can assume $k$ is sufficient large and $|y|\thickapprox B$ which is enough small.
Hence, $\left(\frac{\partial }{\partial y}\right)^2\Phi(y,s,\delta) \neq0$ and it  implies that  there exists a smooth solution $\tilde{y}=\tilde{y}(s,\delta)$
of the equation
$
\frac{\partial }{\partial y}\Phi(y,s,\delta) = 0. $
Write $\tilde{\Phi}(s,\delta)=\Phi(\tilde{y},s,\delta)$.
In \cite[p.29]{LIwenjuan}, we know the phase function can be written as
\begin{equation}\label{ineq-li01}
-t^{m}\xi_2\tilde{\Phi}(s,\delta)=\left(\frac{1}{m^{\frac{m}{m-1}}}-\frac{1}{m^{\frac{1}{m}}}\right)
\left(-\frac{m\xi_1^m}{\xi_2}\right)^{\frac{1}{m-1}}-\frac{\delta^nb^{(n)}(0)}{t^nn!}\left(-\frac{\xi_1}{\xi_2^{\frac{n+1}{n+m}}}\right)^{\frac{m+n}{m-1}}+R(t,\xi,\delta),
\end{equation}
where $R(t,\xi,\delta)$ is homogeneous of degree one in $\xi$ and has at least $n+1$ power of $\delta$.

By using the method of stationary phase, we have
\begin{equation}\label{main-fenjie--2}
\widehat{d\mu_k}(\delta_t\xi)=e^{-it^m\xi_2\tilde{\Phi}(s,\delta)}\chi_{k}
\left(s\right)  \frac{B_k(\delta_t\xi)}{(1+|\delta_t\xi|)^{\frac12}}  +B_k^{err}(\delta_t\xi),
 \end{equation}
where $\chi_k$    is a smooth function supported in the interval $[c_k,\tilde{c}_k]$ ,
for certain non-zero positive constants $c_1\leq c_k<\tilde{c}_k\leq c_2$
deponding only on $k$.
$\{B_k(\xi)\}_k$ is contained in a bounded subset of symbols of order zero.
More precisely,
\begin{equation} \label{estimate-mainterm--2}
|D_{\xi}^{\alpha} B_k(\xi)|\leq C_{\alpha}  (1+|\xi|) ^{-|\alpha|}
\end{equation}
where $C_{\alpha}$ do not depend on $k$.
Furthermore, $B_k^{err}$ is a remainder term and satisfies
$$
 |D_{\xi}^{\alpha} B_k^{err}(\xi)|\leq C_{\alpha,N}  (1+|\xi|) ^{-N},
$$
where $C_{\alpha,N}$   are constants do not depend on $k$.

%First, let us consider the remainder part of  (\ref{main-fenjie--2}).
%Set  $\mathcal{B}_k^{err}=\{B_{k,t}^{err}\}_{t>0}$, where $$
%B_{k,t}^{err}f(x)=\int_{\mathbb{R}^2} e^{i\xi\cdot x} A_k^{err} (\delta_t\xi) \widehat{f}(\xi)d\xi. $$
%By Lemma \ref{lemma-short-remainder}, it is easy to get
%$$
%\|V^{sh}_q(\mathcal{B}_k^{err})f \|_{L^p(\mathbb{R}^2)} \leq C_{p} \|f\|_{L^p(\mathbb{R}^2)}
%$
%for $q\geq2$.
Choose $\beta_k$, $k=0,1,2,\ldots$  as before and
denote
\begin{equation} \label{express-mainterm--2}
B_{j,t}^kf(x):=\int_{\mathbb{R}^2}e^{i\xi\cdot x}e^{-it^m\xi_2\tilde{\Phi}(s,\delta)}\chi_{k}
                    \left(s\right) \beta_{j}(\delta_t\xi) \frac{B_k(\delta_t\xi)}{(1+|\delta_t\xi|)^{\frac12}}\widehat{f}(\xi)d\xi.
 \end{equation}
%and $\mathcal{B}_{main}^{k}=\{B_{main,t}^k\}_{t>0}$.
%$B_{main,t}^kf(x):=\int_{\mathbb{R}^2}e^{i\xi\cdot x}A_{main}^k(\delta_t\xi)\widehat{f}(\xi)d\xi$
%$B_{main}^k(\delta_t\xi) = e^{-it^m\xi_2\tilde{\Phi}(s,\delta)}\chi_{k}
%                    \left(\frac{\xi_1}{t^{m-1}\xi_2}\right)  \frac{A_k(\delta_t\xi)}{(1+|\delta_t\xi|)^{\frac12}}.$
Using the same reason  as in Theorem \ref{Th-isotropic}, we reduce the problem to the task of proving
$$
\sum_{k}2^{-k}\sum_{j\geq1}\left\|V_q^{sh}(\mathcal{B}_{j}^k)\right\|_{L^p\rightarrow L^p}     \leq C_{p,q}
$$
  for $q>\max\{2,p/2\}$ and $p>2$,  where  $\mathcal{B}_k^j=\{B_{j,t}^k\}_{t>0}$.

Denote $B_{j,main}^{k}(\delta_t\xi)=e^{-it^m\xi_2\tilde{\Phi}(s,\delta)}\chi_{k}
                    \left(s\right) \beta_{j}(\delta_t\xi) \frac{B_k(\delta_t\xi)}{(1+|\delta_t\xi|)^{\frac12}}$.
It can be deduced from
$|\delta_t\xi|\thickapprox 2^j$ that
\begin{equation}\label{estimate-0901-2}
|B_{j,main}^{k}(\delta_t\xi)|\leq C 2^{-\frac{j}{2}}.
\end{equation}

 Moreover,
\begin{eqnarray} \label{0729003--2}
&&t\frac{\partial}{\partial t}\left(B_{j,main}^{k}(\delta_t\xi)\right)\nonumber\\
&&=e^{-it^m\xi_2\tilde{\Phi}(s,\delta)}\left(-it\frac{\partial}{\partial t}\left(-t^m\xi_2\Phi(\tilde{y}(s),s,\delta)\right)\chi_k\left(s\right)   \frac{B_k(\delta_t\xi)}{(1+|\delta_t\xi|)^{\frac12}}\right)\nonumber \\
&&\,\,\,\,\,\,\,\,+e^{-it^m\xi_2\tilde{\Phi}(s,\delta)}t\frac{\partial}{\partial t}
\left(\chi_k\left(s\right)\right)\frac{B_k(\delta_t\xi)}{(1+|\delta_t\xi|)^{\frac12}}\nonumber\\
&&\,\,\,\,\,\,\,\,+e^{-it^m\xi_2\tilde{\Phi}(s,\delta)}
\chi_k\left(s\right)t\frac{\partial}{\partial t}\left(\frac{B_k(\delta_t\xi)}{(1+|\delta_t\xi|)^{\frac12}}\right).
\end{eqnarray}
Note that $|\delta_t\xi|\thickapprox |t\xi_1|\thickapprox|t^m\xi_2|\thickapprox 2^j$,
(\ref{ineq-li01}), (\ref{express-mainterm--2}), (\ref{estimate-mainterm--2}) ensure

$$\left|-it\frac{\partial}{\partial t}\left(-t^m\xi_2\tilde{\Phi}(s,\delta)\right)\right|\lesssim \delta^n2^{\frac{j}{2}},$$
$$
\left|t\frac{\partial}{\partial t} \left(
\frac{B_k(t\xi)}{(1+|t\xi|)^{\frac12}}\right)\right| \lesssim 2^{-\frac{j}{2}}, $$
and
$$
\left|t\frac{\partial}{\partial t}\left(\chi_k\left(s\right)\right)\right|\thickapprox 1.
$$
Thus,
\begin{equation} \label{estimate-090101}
\left|t\frac{\partial}{\partial t}\left(B_{j,main}^{k}(\delta_t\xi)\right)\right|\leq C\left(
\delta^{n}2^{\frac{j}{2}}+2^{-\frac{j}{2}}\right).
\end{equation}

The following lemma can be deduced from (\ref{estimate-0901-2}) and (\ref{estimate-090101}).
\begin{lemma}\label{lem-24-0909}
We have
\begin{eqnarray}
 &&\left\|V_{\infty}^{sh}(\mathcal{B}_{j}^kf)\right\|_{L^{\infty}(\mathbb{R}^2)}  \leq C
\|f\|_{L^{\infty}(\mathbb{R}^2)},\label{ineq-00--2-0903}\\
 &&\left\|V_2^{sh}(\mathcal{B}_{j}^kf)\right\|_{L^2(\mathbb{R}^2)}  \leq C (2^{-\frac{n}{2}\cdot k}+2^{-\frac{j}{2}})
\|f\|_{L^2(\mathbb{R}^2)},\label{ineq-22--2-0903}\\
&&\left\|V_2^{sh}(\mathcal{B}_{j}^kf)\right\|_{L^4(\mathbb{R}^2)} \leq C_{\epsilon_1,\epsilon_2}2^{j\epsilon_1}2^{k\epsilon_2}\|f\|_{L^4(\mathbb{R}^2)},             \label{ineq-24--2-0903}\\
&&\|V_4^{sh}(\mathcal{B}_j^kf)\|_{L^4(\mathbb{R}^2)}\leq C_{\epsilon_3,\epsilon_4} 2^{j\left(-\frac{1}{8}+\epsilon_3\right)}\cdot 2^{kn\left(\frac{1}{8}+\epsilon_4\right)}\|f\|_{L^4(\mathbb{R}^2)},\,\,\,\,\,\,\,\,\mbox{for $j>9kn$},\label{ineq-44--1-0903}\\
&&\|V_4^{sh}(\mathcal{B}_j^kf)\|_{L^4(\mathbb{R}^2)}\leq C 2^{-\frac{nk}{2}}\|f\|_{L^4(\mathbb{R}^2)},\,\,\,\,\,\,\,\,\mbox{for $j\leq9kn$}.\label{ineq-44--2-0903}
\end{eqnarray}
\end{lemma}

Using (\ref{ineq-00--2-0903})  and (\ref{ineq-24--2-0903}),
interpolation between $\|V_{\infty}^{sh}(\mathcal{B}_{j}^kf)\|_{L^\infty}$  and $\|V_{2}^{sh}(\mathcal{B}_{j}^kf)\|_{L^4}$ implies
\begin{equation}\label{equ-24090401}
\left\|V_{\frac{p}{2}}^{sh}(\mathcal{B}_j^kf)  \right\|_{L^p}\leq  C_{\epsilon_1,\epsilon_2}2^{j\theta\epsilon_1}2^{k\theta\epsilon_2} \|f\|_{L^p}.
\end{equation}
Using  (\ref{ineq-44--1-0903}) and (\ref{equ-24090401}),
interpolation between $\|V_{\frac{p_1}{2}}^{sh}(\mathcal{B}_{j}^kf)\|_{L^{p_1}}$  and $\|V_{4}^{sh}(\mathcal{B}_{j}^kf)\|_{L^4}$ implies, for $j>9nk$ and $p>4$,
\begin{equation}
\left\|V_{q}^{sh}(\mathcal{B}_j^kf)  \right\|_{L^p}\leq  C_{\epsilon_3,\epsilon_4}2^{j\theta\left(-\frac{1}{8}+\epsilon_3\right)} 2^{kn\theta\left(\frac18+\epsilon_4\right)}\|f\|_{L^p},
\end{equation}
where
$\frac{1}{p}=\frac{1-\theta}{p_1}+\frac{\theta}{4}$ and $\frac{1}{q}=\frac{2(1-\theta)}{p_1}+\frac{\theta}{4}$.
It follows that
$$
\sum_{k}\sum_{j>9nk}2^{-k}\left\|V_{q}^{sh}(\mathcal{B}_j^kf)  \right\|_{L^p}\leq C_{\epsilon_3,\epsilon_4}\sum_{k}2^{-k+kn\theta(-1+\epsilon_3+\epsilon_4)}\|f\|_{L^p}\leq C\|f\|_{L^p},
$$
for $p>4$ and $q>\frac{p}{2}$.
Using  (\ref{ineq-44--2-0903}) and (\ref{equ-24090401}),
interpolation between $\|V_{\frac{p_1}{2}}^{sh}(\mathcal{B}_{j}^kf)\|_{L^{p_1}}$  and $\|V_{4}^{sh}(\mathcal{B}_{j}^kf)\|_{L^4}$ implies, for $j\leq 9nk$ and $p>4$,
\begin{equation}
\left\|V_{q}^{sh}(\mathcal{B}_j^kf)  \right\|_{L^p}\leq  C 2^{-\frac{n\theta k}{2}} \|f\|_{L^p},
\end{equation}
where
$\frac{1}{p}=\frac{1-\theta}{p_1}+\frac{\theta}{4}$ and $\frac{1}{q}=\frac{2(1-\theta)}{p_1}+\frac{\theta}{4}$.
It follows that
\begin{equation}
\sum_{k}\sum_{j\leq 9nk}2^{-k}\left\|V_{q}^{sh}(\mathcal{B}_j^kf)  \right\|_{L^p}\leq C_{\epsilon_3,\epsilon_4}\sum_{k}9nk2^{-k-\frac{n\theta k}{2}}\|f\|_{L^p}\leq C\|f\|_{L^p},
\end{equation}
for $p>4$ and $q>\frac{p}{2}$.

Using (\ref{ineq-22--2-0903})  and (\ref{ineq-24--2-0903}),
interpolation between $\|V_{2}^{sh}(\mathcal{B}_{j}^kf)\|_{L^2}$  and $\|V_{2}^{sh}(\mathcal{B}_{j}^kf)\|_{L^4}$ implies
\begin{equation}\label{equ-24090401-000}
\left\|V_{2}^{sh}(\mathcal{B}_j^kf)  \right\|_{L^p}\leq  C_{\epsilon_1,\epsilon_2} 2^{\left(2-\frac4p\right)\epsilon_1j}2^{\left(2-\frac4p\right)\epsilon_2k} \|f\|_{L^p}.
\end{equation}
Using (\ref{ineq-44--1-0903})  and (\ref{equ-24090401-000}),
interpolation between $\|V_{2}^{sh}(\mathcal{B}_{j}^kf)\|_{L^{p_1}}$  and $\|V_{4}^{sh}(\mathcal{B}_{j}^kf)\|_{L^4}$ implies,
for $j> 9nk$ and $2<p\leq 4$,
\begin{equation}
\left\|V_{q}^{sh}(\mathcal{B}_j^kf)  \right\|_{L^p}\leq  C_{\epsilon_1,\epsilon_2,\epsilon_3,\epsilon_4} 2^{\left(2-\frac4p\right)\epsilon_1(1-\theta)j}
2^{\left(2-\frac4p\right)\epsilon_2(1-\theta)k}2^{\theta j(-\frac18+\epsilon_3)}2^{\theta kn(\frac18+\epsilon_4)} \|f\|_{L^p},
\end{equation}
where
$\frac{1}{q}=\frac{1-\theta}{2}+\frac{\theta}{4}$ and $\frac{1}{p}=\frac{1-\theta}{p_1}+\frac{\theta}{4}$.
It follows that
$$
\sum_{k}2^{-k}\sum_{j> 9nk} \leq C_{\epsilon_1,\epsilon_2,\epsilon_3,\epsilon_4} 2^{\left(2-\frac4p\right)\epsilon_1(1-\theta)j}
2^{\left(2-\frac4p\right)\epsilon_2(1-\theta)k}2^{-k+\theta nk(-1+\epsilon_3+\epsilon_4)}\leq C \|f\|_{L^p},
$$
for $2<p\leq4$ and $q>2$.
Using (\ref{ineq-44--2-0903})  and (\ref{equ-24090401-000}),
interpolation between $\|V_{2}^{sh}(\mathcal{B}_{j}^kf)\|_{L^{p_1}}$  and $\|V_{4}^{sh}(\mathcal{B}_{j}^kf)\|_{L^4}$ implies,
for $j> 9nk$ and $2<p\leq 4$,
\begin{equation}
\left\|V_{q}^{sh}(\mathcal{B}_j^kf)  \right\|_{L^p}\leq  C_{\epsilon_1,\epsilon_2} 2^{\left(2-\frac4p\right)\epsilon_1(1-\theta)j}
2^{\left(2-\frac4p\right)\epsilon_2(1-\theta)k}
2^{-\theta\frac{kn}{2}} \|f\|_{L^p},
\end{equation}
where
$\frac{1}{q}=\frac{1-\theta}{2}+\frac{\theta}{4}$ and $\frac{1}{p}=\frac{1-\theta}{p_1}+\frac{\theta}{4}$.
It follows that
$$
\sum_{k}2^{-k}\sum_{j\leq 9nk} \leq C_{\epsilon_1,\epsilon_2}9nk 2^{\left(2-\frac4p\right)\epsilon_1(1-\theta)j}
2^{\left(2-\frac4p\right)\epsilon_2(1-\theta)k}2^{-k-\theta\frac{nk}{2}}\leq C \|f\|_{L^p},
$$\vspace{20pt}
for $2<p\leq4$ and $q>2$.

\textit{ Proof of Lemma \ref{lem-24-0909}.}
(\ref{ineq-00--2-0903}) has been obtained in \cite{LIwenjuan}(see the proof of (2.52)).
By using Lemma \ref{Sobolev embbding}, we see that  $\|V_{2,1}(\mathcal{B}_j^kf)\|_{L^2}$ is dominated by
\begin{equation} \label{V-1-estimate--2}
\left\|\int_{\mathbb{R}^2}e^{ix\cdot\xi}
                                        B_{j,main}^{k}(\delta_t\xi)
                                         \widehat{f}(\xi)d\xi\right\|^{\frac{1}{2}}_{L^2}  \times \left\|\int_{\mathbb{R}^2}e^{ix\cdot\xi}t\frac{\partial}{\partial t}\left(B^k_{j,main}(\delta_t\xi)\right)\widehat{f}(\xi)d\xi\right\|^{\frac{1}{2}}_{L^2}, \end{equation}
where we are taking the $L^2$  norm with respect to $\mathbb{R}^2\times[1,2]$.
As in the proof of Lemma 2.6 in \cite{LIwenjuan}, we have $\|V_{2,1}(\mathcal{B}_j^kf)\|_{L^2}\leq C \left(2^{-\frac{nk}{2}}+2^{-\frac{j}{2}}\right)
\|f\|_{L^2}$.
As in the argument for (\ref{ineq-0830-001}), (\ref{ineq-22--2-0903})
is easy to get by using Littlewood-Paley theorem.

Using Lemma \ref{Sobolev embbding} again, we see that
$\|V_{4,1}(\mathcal{B}_j^kf)\|_{L^4}$ is dominated by
\begin{equation} \label{V-1-estimate--2}
\left\|\int_{\mathbb{R}^2}e^{ix\cdot\xi}
                                        B_{j,main}^{k}(\delta_t\xi)
                                         \widehat{f}(\xi)d\xi\right\|^{\frac{3}{4}}_{L^4}  \times \left\|\int_{\mathbb{R}^2}e^{ix\cdot\xi}t\frac{\partial}{\partial t}\left(B^k_{j,main}(\delta_t\xi)\right)\widehat{f}(\xi)d\xi\right\|^{\frac{1}{4}}_{L^4}, \end{equation}
where we are taking the $L^4$  norm with respect to $\mathbb{R}^2\times[1,2]$.
As in the proof  of inequality (2.78) in \cite{LIwenjuan}, we have $\|V_{4,1}(\mathcal{B}_j^kf)\|_{L^2}\leq C_{\epsilon_3,\epsilon_4} 2^{j\left(-\frac{1}{8}+\epsilon_3\right)}\cdot 2^{kn\left(\frac{1}{8}+\epsilon_4\right)}\|f\|_{L^4(\mathbb{R}^2)}$.
Then it can be deduced that (\ref{ineq-44--1-0903}) is valid.

(\ref{ineq-44--2-0903}) will be obtained by employing the M.Riesz interpolation theorem between (\ref{ineq-00--2-0903}) and (\ref{ineq-22--2-0903}).
As in the argument for (\ref{ineq-0830-001}), (\ref{ineq-22--2-0903})
is easy to get by using the  Littlewood-Paley theorem.

It is remained  for us to prove (\ref{ineq-24--2-0903}).
By Lemma \ref{Sobolev embbding},
\begin{equation*}
 V_{2}^{sh}(\mathcal{B}_j^kf)(x)\leq  C \left(\int_{\mathbb{R}}\left|B_{j,t}^kf(x) \right|^2 \frac{dt}{t} \right)^{\frac{1}{4}}
  \left(\int_{\mathbb{R}}\left|t\frac{\partial}{\partial t}B_{j,t}^kf(x) \right|^2 \frac{dt}{t} \right)^{\frac{1}{4}}. \end{equation*}
Thus it follows that $\|V_2^{sh}(\mathcal{B}_j^kf)\|_{L^4}$   is bounded by
\begin{eqnarray}\label{ineq-240909}
&&\left\|\left(\int_{\mathbb{R}}\left|\int_{\mathbb{R}^2}e^{ix\cdot\xi}B_{main}^{k}(\delta_t\xi)\beta_j(|\delta_t\xi|)\widehat{f}(\xi)d\xi \right|^2 \frac{dt}{t} \right)^{\frac{1}{2}} \right\|_{L^4(dx)}^{\frac12} \times\nonumber\\
&&\left\|\left(\int_{\mathbb{R}}\left|\int_{\mathbb{R}^2}e^{ix\cdot\xi}t\frac{\partial}{\partial t}\left(B^k_{main}(\delta_t\xi)\beta_j(|\delta_t\xi|)\right)\widehat{f}(\xi)d\xi \right|^2 \frac{dt}{t} \right)^{\frac{1}{2}} \right\|_{L^4(dx)}^{\frac12}.
\end{eqnarray}
Moreover,
\begin{eqnarray} \label{0729003--2}
&&t\frac{\partial}{\partial t}\left(B^k_{main}(\delta_t\xi)\beta_j(|\delta_t\xi|)\right)\nonumber\\
&&=e^{-it^m\xi_2\Phi(\tilde{y}(s),s,\delta)}\left(-it\frac{\partial}{\partial t}\left(-t^m\xi_2\Phi(\tilde{y}(s),s,\delta)\right)\chi_k\left(\frac{\xi_1}{t^{m-1}\xi_2}\right)   \frac{B_k(\delta_t\xi)}{(1+|\delta_t\xi|)^{\frac12}}\right)\nonumber \\
&&\,\,\,\,\,\,\,\,+e^{-it^m\xi_2\Phi(\tilde{y}(s),s,\delta)}t\frac{\partial}{\partial t}
\left(\chi_k\left(\frac{\xi_1}{t^{m-1}\xi_2}\right)\right)\frac{B_k(\delta_t\xi)}{(1+|\delta_t\xi|)^{\frac12}}\nonumber\\
&&\,\,\,\,\,\,\,\,+e^{-it^m\xi_2\Phi(\tilde{y}(s),s,\delta)}
\chi_k\left(\frac{\xi_1}{t^{m-1}\xi_2}\right)t\frac{\partial}{\partial t}\left(\frac{B_k(\delta_t\xi)}{(1+|\delta_t\xi|)^{\frac12}}\right).
\end{eqnarray}
From  (\ref{ineq-li01}), (\ref{express-mainterm--2}), (\ref{estimate-mainterm--2}) and $|\delta_t\xi|\thickapprox |t\xi_1|\thickapprox|t^m\xi_2|\thickapprox 2^j$,
 we get
$|-it\frac{\partial}{\partial t}\left(-t^m\xi_2\Phi(\tilde{y}(s),s,\delta)\right)|\lesssim \delta^n2^{j}+2^{-\frac{j}{2}}$,
$
\left|t\frac{\partial}{\partial t} \left(
\frac{B_k(t\xi)}{(1+|t\xi|)^{\frac12}}\right)\right| \lesssim 2^{-\frac{j}{2}}, $
and
$
\left|t\frac{\partial}{\partial t}\left(\chi_k\left(\frac{\xi_1}{t^{m-1}\xi_2}\right)\right)\right|\thickapprox 1.
$

We conclude $\|V_2^{sh}(\mathcal{B}_j^kf)\|_{L^4}$ is dominated by
\begin{equation} \label{estimate-squrefuncion--2} \left(2^{-\frac{nk}{2}}+2^{-\frac{j}{2}}\right)\left\|\left(\int_{\mathbb{R}}|\mathcal{Q}_{k,j}f(\cdot,t)|^2\frac{dt}{t}\right)^{\frac12}\right\|_{L^4}, \end{equation}
where
$$
 \mathcal{Q}_{k,j}f(x,t)= \int_{\mathbb{R}^2}e^{ix\cdot\xi}e^{it^m\xi_2\Phi(\tilde{y}(s),s,\delta)}a_k(\delta_t\xi)\beta_j(|\delta_t\xi|)\widehat{f}(\xi)d\xi
$$
and $a_k(\xi)$ is a symbol of order zero in $\xi$.
More precisely,
\begin{equation}
|D_{\xi}^{\alpha}a_k(\xi)| \leq C_{\alpha}  (1+|\xi|)^{-|\alpha|}, \end{equation}
where $C_{\alpha}$ do not depend on $k$.

We will prove the following theorem in section 5.
\begin{theorem} \label{estimate-important--2}
Suppose $q(\xi)$ is homogeneous of degree one and
$\beta\in C_0^{\infty}(\mathbb{R})$ such that
$\supp \beta\subset (1/2,4)$ and $\beta(r)=1$ for
$1\leq r\leq2$.
Let
$$\mathcal{Q}_{\lambda,\delta}f(x,t)=\int_{\mathbb{R}^2}e^{i\left(x\cdot\xi+t\delta q(\xi)+\delta^{2}R(\xi,t,\delta)\right)}a(\xi,t)
\beta(\lambda^{-1}|\delta_t\xi|)\widehat{f}(\xi)d\xi,$$
where $a(\xi,t)$ is a symbol of order zero in $\xi$
and $R(\xi,t,\delta)$ is homogeneous of degree one in $\xi$.
Then,
for all $j>1$ and $\lambda>\delta^{-2}$,
\begin{equation}\label{ineq-082102}
\left\|\mathcal{Q}_{\lambda,\delta}f(x,t) \right\|_{L_x^4(\mathbb{R}^2,L_t^2([1,2]))}
\leq  C_{\epsilon,\epsilon'}\lambda^{\epsilon}\delta^{-\frac12-\epsilon'} \|f\|_{L^4(\mathbb{R}^2)}.
\end{equation}
\end{theorem}

Using Theorem \ref{estimate-important--2}, Lemma \ref{estimate-squrefuncion} and  (\ref{estimate-squrefuncion--2}),
we have, for $j>2nk$,
\begin{equation}   \label{inequ-24--2}
\|V_2^{sh}(\mathcal{B}_j^kf)\|_{L^4}\leq C_{\epsilon_1,\epsilon_2}2^{j\epsilon_1}2^{k\epsilon_2} \|f\|_{L^4}.
\end{equation}

When $j\leq 2k$, by using Plancherel theorem, we  obtain
\begin{equation}\label{equ-09090001}
\left\|\left(\int_{1}^2|\mathcal{Q}_{k,j}f(\cdot,t)|^2dt\right)^{\frac12}\right\|_{L^2}\leq C\|f\|_{L^2}.
\end{equation}
As in the proof of (2.55) in \cite{LIwenjuan}, we have
\begin{equation}\label{equ-09090002}
\left\|\sup_{1\leq t\leq 2}|\mathcal{Q}_{k,j}f(\cdot,t)|\right\|_{L^{\infty}}\leq C 2^{\frac{j}{2}}\|f\|_{L^{\infty}}.
\end{equation}
Employing the M. Riesz interpolation theorem between (\ref{equ-09090001}) and (\ref{equ-09090002}),
we have
$$
\left\|\left(\int_{1}^2|\mathcal{Q}_{k,j}f(\cdot,t)|^4dt\right)^{\frac14}\right\|_{L^4}\leq C 2^{\frac{j}{4}} \|f\|_{L^4}.
$$
Thus by H\"{o}lder inequality and (\ref{ineq-240909}), we get (\ref{ineq-24--2-0903}) for $j\leq 2nk$.

\subsection{The proof of the necessary condition}

We now turn to prove the necessary  part of the statement in  Theorem \ref{Th-nonisotropic}  and show that the a priori
inequality
$$
\|V_q(\mathcal{B}f)\|_{L^p}\leq C_p\|f\|_{L^p}
$$
implies that $q\geq p/2$.

\subsubsection{Preliminary work}
We choose $y=\frac{1}{m}y^m\left(1+y^nb(y)\right)$, where $b(0)=1$.
Then
$$
\widehat{d\mu}(\delta_t\xi)=\int e^{-it^m\xi_2\Phi(y,s)}\eta(y)dy,
$$
where $\Phi(y,s)=-sy+\frac{1}{m}y^m(1+y^nb(y))$ with $s=-\frac{\xi_1}{t^{m-1}\xi_2}$.

Note that
$$
(\partial_1)^2\Phi(0,s)=\cdots=(\partial_1)^{m-1}\Phi(0,s)=0,\,\,(\partial_1)^{m}\Phi(0,s)=(m-1)!\neq 0,
$$
and  $|y|$ can be assume  enough small. Thus  when $y\neq0$, we have $(\partial_1)^2\Phi(y,s)\neq 0$ and we
could
obtain $\tilde{y}=\tilde{y}(s)$ such that
$$
\partial_1\Phi(y,s)=-s+y^{m-1}+\frac{m+n}{m}y^{m+n-1}b(y)+\frac{1}{m}y^{m+n}b'(y)=0.
$$

Let $L=s^{\frac{1}{m-1}}$ and $\Psi(y,L)=\Phi(y,L^{m-1})$. It follows that $y^*(L)=\tilde{y}(L^{m-1})$ is the solution of the equation
$$
(\partial_1\Psi)(y,L)=-L^{m-1}+y^{m-1}+\frac{m+n}{m}y^{m+n-1}b(y)+\frac{1}{m}y^{m+n}b'(y)=0.
$$
It can be deduced that
\begin{equation}\label{ineq-24091201}
y^*(0)=0.
\end{equation}

Denote $\tilde{\Psi}_1(L)=(\partial_1\Psi)(y^*(L),L)$.
Then
$$
(D_L)^{m-1}\tilde{\Psi}_1(L)\left.\right|_{L=0}=-(m-1)!+(m-1)!\left((y^*)'(0)\right)^{m-1}=0.
$$
It follows that
\begin{equation}\label{ineq-24091202}
(y^*)'(0)=1.
\end{equation}

Moreover,
\begin{eqnarray*}
&&(D_L)^{m+n-1}\tilde{\Psi}_1(L)\left.\right|_{L=0}\\
&&=(m-1)!\cdot(m-1)\left((y^*)'(0)\right)^{m-2}(y^*)^{(n+1)}(0)+\frac{(m+n)!}{m}\left((y^*)'(0)\right)^{m+n-1}b(0)=0.
\end{eqnarray*}
It can be deduced that
\begin{equation}\label{ineq-24091203}
(y^*)^{(n+1)}(0)=-\frac{(m+n)!}{(m-1)m!}.
\end{equation}

Denote $\tilde{\Psi}(L)=\Psi(y^*(L),L)$.
Then (\ref{ineq-24091201}), (\ref{ineq-24091202}) and (\ref{ineq-24091203}) ensure that
$$
\tilde{\Psi}(0)=D(\tilde{\Psi})(0)=\cdots=D^{m-1}(\tilde{\Psi})(0)=0,\,\,\,\,\,\,\,\,D^{m}(\tilde{\Psi})(0)=-(m-1)!,
$$
and
$$
D^{m+1}\tilde{\Psi}(0)=\cdots=D^{m+n-1}(\tilde{\Psi})(0)=0,\,\,\,\,D^{m+n}(\tilde{\Psi})(0)=\frac{(m+n)!}{m(m-1)}.
$$
By Taylor expansion,  we have
$$
\tilde{\Psi}(L)=-\frac{1}{m}L^m+\frac{1}{m(m-1)}L^{m+n}+\mathcal{O}(L^{m+n+1}).
$$
It can be deduced that
$$
\tilde{\Phi}(s)=-\frac{1}{m}s^{\frac{m}{m-1}}+\frac{1}{m(m-1)}s^{\frac{m+n}{m-1}}+\mathcal{O}(s^{\frac{m+n+1}{m-1}}).
$$

Using the method of stationary phase, we have
\begin{eqnarray*}
&&\widehat{d\mu}(\delta_t\xi)\\
&&=\exp\left(i\left(-\frac{1}{m}\cdot\frac{\xi_1^{\frac{m}{m-1}}}{\xi_2^{\frac{1}{m-1}}}+
\frac{1}{m(m-1)}\frac{\xi_1^{\frac{m+n}{m-1}}}{t^n\xi_2^{\frac{n+1}{m-1}}}+\mathcal{O}\left(\frac{\xi_1^{\frac{m+n+1}{m-1}}}{t^{n+1}\xi_2^{\frac{n+2}{m-1}}}\right)\right)\right)
\chi\left(s\right)a(\delta_t\xi)\\
&&\,\,\,\,\,\,\,\,\,\,\,\,+B(\delta_t\xi),
\end{eqnarray*}
where $\chi$ is a smooth function supported in a small neighbourhood near zero and
$$|D_\xi^{\alpha} a(\xi)|\leq C_{\alpha}(1+|\xi|)^{-1/2-|\alpha|}$$
and  $$|D_\xi^{\alpha} B(\xi)|\leq C_{\alpha,N}(1+|\xi|)^{-N}.$$

Thus, we can write
$$
B_tf(x)=B_{1,t}f(x)+B_{2,t}f(x),
$$
where
\begin{eqnarray*}
&&B_{1,t}f(x)\\
&&=\frac{1}{(2\pi)^2}\int_{\mathbb{R}^2}\exp\left(i\left(-\frac{1}{m}\cdot\frac{\xi_1^{\frac{m}{m-1}}}{\xi_2^{\frac{1}{m-1}}}+
\frac{1}{m(m-1)}\frac{\xi_1^{\frac{m+n}{m-1}}}{t^n\xi_2^{\frac{n+1}{m-1}}}+\mathcal{O}
\left(\frac{\xi_1^{\frac{m+n+1}{m-1}}}{t^{n+1}\xi_2^{\frac{n+2}{m-1}}}\right)\right)\right)\\
&&\,\,\,\,\,\,\,\,\,\,\,\,\times
\chi\left(s\right)a(\delta_t\xi)\widehat{f}(\xi)e^{ix\cdot\xi}d\xi,
\end{eqnarray*}
and
$$
B_{2,t}f(x)=\frac{1}{(2\pi)^2}\int_{\mathbb{R}^2}e^{ix\cdot\xi}B(\delta_t\xi)\widehat{f}(\xi)d\xi.
$$

Let $\eta_1=\frac{1}{m(m-1)}\frac{\xi_1^{\frac{m+n}{m-1}}}{\xi_2^{\frac{n+1}{m-1}}}$
and $\eta_2=\frac{\xi_1}{\xi_2}$.
%\begin{equation*}\label{general-gn2}
%\left\{\begin{array}{l}
%\eta_1=\frac{1}{m(m-1)}\frac{\xi_1^{\frac{m+n}{m-1}}}{\xi_2^{\frac{n+1}{m-1}}}\,,\\
%\eta_2=\frac{\xi_1}{\xi_2}\,.
%\end{array}\right.
%\end{equation*}
%Then
%\begin{equation*}\label{general-gn22}
%\left\{\begin{array}{l}
%\xi_1=\frac{1}{m(m-1)}\frac{\eta_1}{\eta_2^{\frac{n+1}{m-1}}}\,,\\
%\xi_2=\frac{1}{m(m-1)}\frac{\eta_1}{\eta_2^{\frac{m+n}{m-1}}}\,.
%\end{array}\right.
%\end{equation*}
Then after the change of variable, we have
\begin{eqnarray*}
B_{1,t}f(x)&=&\frac{1}{(2\pi)^2}\int_{\mathbb{R}^2}\exp\left( \frac{i\eta_1x_2}{m(m-1)}\left(\frac{x_1/x_2}{\eta_2^{\frac{n+1}{m-1}}}+
\frac{1}{\eta_2^{\frac{m+n}{m-1}}}\right)\right)\chi(\frac{\eta_2}{t^{m-1}})\\
&&\,\,\,\,\,\,\times \exp\left(i\left(\frac{\eta_1}{t^n}+\mathcal{O}\left(\frac{\eta_1\eta_2^{\frac{1}{m-1}}}{t^{n+1}}\right)\right)\right)\tilde{a}\left(\frac{1}{t^n}\left(t\eta_1,t^{-m+1}\eta_2\right)\right)
\\
&&\,\,\,\,\,\,\,\,\,\,\,\,\times\exp\left({-i\frac{1}{m^2(m-1)}\frac{\eta_1}{\eta_2^{\frac{n}{m-1}}}}\right)\widehat{\tilde{f}}(\eta)\left|
\frac{1}{m^2(m-1)^2}\frac{\eta_1}{\eta_2^{\frac{2(m+n)}{m-1}}}
\right|
d\eta
\end{eqnarray*}
and
\begin{eqnarray*}
B_{2,t}f(x)&=&\frac{1}{(2\pi)^2}\int_{\mathbb{R}^2}\exp\left( \frac{i\eta_1x_2}{m(m-1)}\left(\frac{x_1/x_2}{\eta_2^{\frac{n+1}{m-1}}}+
\frac{1}{\eta_2^{\frac{m+n}{m-1}}}\right)\right)\tilde{B}\left(\frac{1}{t^n}\left(t\eta_1,t^{-m+1}\eta_2\right)\right)\\
&&\,\,\,\,\times
\exp\left({-i\frac{1}{m^2(m-1)}\frac{\eta_1}{\eta_2^{\frac{n}{m-1}}}}\right)\widehat{\tilde{f}}(\eta)
\left|
\frac{1}{m^2(m-1)^2}\frac{\eta_1}{\eta_2^{\frac{2(m+n)}{m-1}}}
\right|
d\eta,
\end{eqnarray*}
where
$$
\tilde{a}\left(\frac{1}{t^n}(t\eta_1,t^{-m+1}\eta_2)\right)=a\left(\frac{t\eta_1}{t^n(t^{-m+1}\eta_2)^{\frac{n+1}{m-1}}},
\frac{t^m\eta_1}{t^n(t^{-m+1}\eta_2)^{\frac{m+n}{m-1}}}\right)
$$
and
$$
\tilde{B}\left(\frac{1}{t^n}(t\eta_1,t^{-m+1}\eta_2)\right)=B\left(\frac{t\eta_1}{t^n(t^{-m+1}\eta_2)^{\frac{n+1}{m-1}}},
\frac{t^m\eta_1}{t^n(t^{-m+1}\eta_2)^{\frac{m+n}{m-1}}}\right).
$$

%Next let us estimate $\tilde{a}\left(\frac{1}{t^n}(t\eta_1,t^{-m+1}\eta_2)\right)$ and $\tilde{B}(\frac{1}{t^n}(t\eta_1,t^{-m+1}\eta_2))$.
%\begin{eqnarray*}
%\frac{\partial}{\partial \eta_1}\left(\tilde{a}(\frac{1}{t^n}(t\eta_1,t^{-m+1}\eta_2))\right)=
%(\partial_1a)(t\xi_1,t^m\xi_2)\cdot \frac{t}{t^n(t^{-m+1}\eta_2)^{\frac{n+1}{m-1}}}
%+(\partial_2a)(t\xi_1,t\xi_2)\cdot \frac{t^m}{t^n(t^{-m+1}\eta_2)^{\frac{m+n}{m-1}}}
%\end{eqnarray*}
%\begin{eqnarray*}
%\frac{\partial}{\partial \eta_2}\left(\tilde{a}(\frac{1}{t^n}(t\eta_1,t^{-m+1}\eta_2))\right)=
%(\partial_1a)(t\xi_1,t\xi_2)\cdot -\frac{n+1}{m-1}\frac{t\eta_1}{\eta_{2}^{\frac{m+n}{m-1}}}
%+(\partial_2a)(t\xi_1,t\xi_2)\cdot -\frac{m+n}{m-1}\frac{t\eta_1}{\eta_{2}^{\frac{2m+n-1}{m-1}}}
%\end{eqnarray*}
Since $\eta_2\thicksim t^{m-1}$ and $t\thicksim1$, we have
%$$
%\left|\frac{\partial}{\partial \eta_1}\left(\tilde{a}\frac{1}{t^n}(t\eta_1,t^{-m+1}\eta_2))\right)\right|\leq
%C(1+|\eta_1|)^{-\frac{3}{2}}
%$$
%and
%$$
%\left|\frac{\partial}{\partial \eta_2}\left(\tilde{a}(\frac{1}{t^n}(t\eta_1,t^{-m+1}\eta_2))\right)\right|\leq
%C(1+|\eta_1|)^{-\frac{1}{2}}.
%$$
%It follows that
$$
\left|\frac{\partial^{l+k}}{(\partial \eta_1)^l(\partial \eta_2)^k}\left(\tilde{a}\left(\frac{1}{t^n}(t\eta_1,t^{-m+1}\eta_2)\right)\right)\right|\leq C_{l,k} (1+|\eta_1|)^{-\frac{1}{2}-l},
$$
%By the same steps, we have
and
$$
\left|\frac{\partial^{l+k}}{(\partial \eta_1)^l(\partial \eta_2)^k}\left(\tilde{B}\left(\frac{1}{t^n}(t\eta_1,t^{-m+1}\eta_2)\right)\right)\right|\leq C_{l,k,N} (1+|\eta_1|)^{-N}.
$$
\subsubsection{Estimate of the lower bound for $\frac{\|V_q(\mathcal{B}f)\|_{L^p}}{\|f\|_{L^p}}$}
Choose $\widehat{\tilde{f}}_{\lambda}(\eta)=\chi_0\left(\frac{\eta_1}{\lambda}\right)\chi_1(\eta_2)
e^{-i\frac{\eta_1^2}{2\lambda}}$ where $\chi_0\left(\eta_1\right)$ and $\chi_1\left(\eta_1\right)$ is supported in a very small interval $[B/2,2B]$ near zero.
Here still applying  the method of stationary phase and we can obtain  $\|f_{\lambda}\|_{L^p}=\mathcal{O}(\lambda)$.

In order to derive the lower bound for $V_q(\mathcal{B}f)(x)$, we also restrict $x$  in the region $0<x_2\leq c_0\lambda^{-1}$ and $c_1\leq \frac{x_1}{x_2}\leq \tilde{c_1}$.
After a change of variable we may write
\begin{equation*}
B_{1,t}f(x)=\frac{\lambda^2}{2\pi}\int_{\mathbb{R}}e^{-i\frac{\lambda\eta_1^2}{2}}
\chi_0\left(\eta_1\right)
G(\lambda\eta_1,x,t)e^{i\frac{\lambda\eta_1}{t^n}}\eta_1
d\eta_1,
\end{equation*}
where
\begin{eqnarray*}
G(\lambda\eta_1,x,t)&=&\frac{1}{2\pi}\int_{\mathbb{R}^2}\exp\left( \frac{i(\lambda x_2)\eta_1}{m(m-1)}\left(\frac{x_1/x_2}{\eta_2^{\frac{n+1}{m-1}}}+
\frac{1}{\eta_2^{\frac{m+n}{m-1}}}\right)\right)\chi\left(\frac{\eta_2}{t^{m-1}}\right)\chi_1(\eta_2)\\
&&\,\,\,\,\,\,\,\,\times\exp\left(i\left(\frac{\lambda\eta_1}{t^n}+\mathcal{O}\left(\frac{\lambda\eta_1\eta_2^{\frac{1}{m-1}}}{t^{n+1}}\right)\right)\right)
\tilde{a}\left(\frac{1}{t^n}(\lambda\eta_1,t^{-m+1}\eta_2)\right)\\
&&\,\,\,\,\,\,\,\,\,\,\,\,\,\,\,\,\times
\exp\left({-i\frac{1}{m^2(m-1)}\frac{\lambda\eta_1}{\eta_2^{\frac{n}{m-1}}}}\right)\frac{1}{m^2(m-1)^2}\frac{1}{\eta_2^{\frac{2(m+n)}{m-1}}}
d\eta_2.
\end{eqnarray*}
Since $0<\lambda x_2\leq c_0$, $c_1\leq \frac{x_1}{x_2}\leq \tilde{c_1}$, $\eta_2\approx1$ and $t\approx 1$, we have
\begin{equation}\label{equ-0915-04}
|G(\lambda\eta_1,x,t)|= C\lambda^{-\frac{1}{2}}\widetilde{G}(\lambda,\eta_1,x,t),
\end{equation}
where $\widetilde{G}(\lambda,\eta_1,x,t)$ is a symbol order of zero in $\eta_1$.

%Let us firstly try to estimate $G(\eta_1,x,t)$.
%In order to use stationary phase theorem, we have
%\begin{equation*}
%\frac{\partial}{\partial\eta_2}\left(\frac{x_1}{\eta_2^{\frac{n+1}{m-1}}}+
%\frac{x_2}{\eta_2^{\frac{m+n}{m-1}}}\right)
%=-\eta_2^{-\frac{m+n}{m-1}}\left(x_1+\frac{m+n}{n-1}\eta_2^{-1}x_2\right).
%\end{equation*}
%Then $\tilde{\eta}_2=-\frac{m+n}{n+1}\frac{x_2}{x_1}$ is the solution of the equation
%$$
%\frac{\partial}{\partial\eta_2}\left(\frac{1}{m(m-1)}\frac{\eta_1}{\eta_2^{\frac{n+1}{m-1}}}x_1+
%\frac{1}{m(m-1)}\frac{\eta_1}{\eta_2^{\frac{m+n}{m-1}}}x_2\right)=0.
%$$
%Using the method of  stationary phase and  $t\sim 1$,
%it follows that
%$$
%G(\eta_1,x,t)=c_0\exp\left(-i\frac{(n+1)^{\frac{n+1}{m-1}}}{m(m+n)^{\frac{m+n}{m-1}}}\frac{(-x_1)^{\frac{m+n}{m-1}}}{x_2^{\frac{m+1}{m-1}}}\right)
%\chi(\frac{\tilde{\eta_2}}{t^{m-1}})\frac{M(t,\eta_1)}{1+\eta_1}
%\chi_1(\tilde{\eta_2})
%\cdot\frac{1}{\tilde{\eta}_2^{\frac{2(m+n)}{m-1}}}
%$$
%where $M(t,\eta_1)$ is a symbol of order zero in $\eta_1$.

Let $\tilde{t}=\frac{1}{t^n}$. Then
\begin{equation}
B_tf_{\lambda}(x)
=\frac{\lambda^2}{2\pi}e^{i\cdot\frac{\lambda \tilde{t}^2}{2}}\int_{\mathbb{R}}e^{-i\cdot\frac{\lambda}{2}(\eta_1-\tilde{t})^2}
\chi_0(\eta_1)G(\lambda\eta_1,x,\tilde{t}^{-\frac{1}{n}})\eta_1d\eta_1 .
\end{equation}
Employing the method of  stationary phase,
we have
\begin{equation*}
B_tf_{\lambda}(x)
=c\lambda e^{i\cdot\frac{\lambda \tilde{t}^2}{2}}
 \tilde{t}\chi_0(\tilde{t})\widetilde{G}(\lambda,\tilde{t},x,\tilde{t}^{-\frac{1}{n}})+\mathcal{O}(1).
\end{equation*}

Denote
\begin{equation*}
I(x,\tilde{t})=c\lambda e^{i\cdot\frac{\lambda \tilde{t}^2}{2}}
 \tilde{t}\chi_0(\tilde{t})\widetilde{G}(\lambda,\tilde{t},x,\tilde{t}^{-\frac{1}{n}}).
\end{equation*}
Choose $$\tilde{t}_n=\left(2+\frac{2n\pi}{\lambda}\right)^{\frac12}.$$
Then $|\tilde{t}_n-\tilde{t}_{n+1}|=\mathcal{O}(\lambda^{-1})$ and  $
|I(x,t_n)-I(x,t_{n+1})| \gtrsim \lambda(c-C\lambda^{-1})
$.
Therefor,
$$
\left(\sum_{1\leq n <\frac{\lambda}{100}}|I(x,t_n)-I(x,t_{n+1})|^q\right)^{\frac1q}\gtrsim \lambda^{1+\frac1q}(c-C\lambda^{-1}).
$$
 It implies
 $$
 V_q(\mathcal{A}f_{\lambda})(x)\gtrsim  \lambda^{1+\frac1q}(c-C\lambda^{-1}).
 $$

Recall that $x_2\leq \lambda^{-1}$ and $c_1x_2\leq x_1\leq \tilde{c}_1x_2$.
Then we obtain $\|V_q(\mathcal{A}f_{\lambda})\|_{L^p}\gtrsim \lambda^{1+\frac1q-\frac2p}$.
Since $\|f_{\lambda}\|_{L^p}\leq C\lambda$,
 this yields the restriction $q\geq p/2$.

\section{The $L^4_x(L^2_t)$-estimate associated to Fourier integral operators not satisfying the cinematic curvature condition uniformly}

In this section, we will prove   Theorem \ref{estimate-important--2}.
The approach presented here is based on the proof of Theorem 6.2 in \cite{1993-MSS-loacal-smoothing}, but it requires several modifications. To establish the $L^4$ estimate, we can also assume that the support of $a(\xi,t)$ lies within a cone with an angle smaller than $\lambda^{-\epsilon}$. This assumption only results in an additional increase in the constants by a factor of $\lambda^{\epsilon}$.

Let $\phi(x,t,\xi,\delta)=x\cdot\xi+t\delta q(\xi)+\delta^{2}R(\xi,t,\delta)$  and $a_{\lambda}(\xi,t)=a(\xi,t)
\beta(\lambda^{-1}|\delta_t\xi|)$.
Now, as in \cite{1993-MSS-loacal-smoothing},  we introduce an angle decomposition $\mathcal{Q}_{\lambda,\delta}=\sum_{v}\mathcal{Q}_{\lambda,\delta}^{\nu}$ with
$$
\mathcal{Q}_{\lambda,\delta}^{\nu}f(x,t)=\int_{\mathbb{R}^2}e^{i\phi(x,t,\xi,\delta)}a_{\lambda}(\xi,t)
\chi_{\nu}(\xi)
\widehat{f}(\xi)d\xi.
$$
Here, the function $\chi_{v}$ has support in a sector with an approximate angle of $\thickapprox \lambda^{-1/2}$ that includes the unit vector $\xi_v$.
It is assumed that the indices $\nu$ are chosen such that
$\arg \xi_{\nu}<\arg{\xi_{\nu+1}}$, implying that $|\xi_{\nu}-\xi_{\nu+1}|\thickapprox\lambda^{-\frac12}$.
 Furthermore, for any two indices $\nu$, $\mu$ in this sum, it holds true that $|\nu-\mu|\leq \lambda^{1/2-\epsilon}$.

 Using the Cauchy-Schwarz inequality, we get immediately
\begin{eqnarray}\label{equ-0920-001}
\|\mathcal{Q}_{\lambda,\delta}f\|_{L^4(L^2)}&\leq& \sum_{\lambda^{\epsilon}\leq2^r\leq \lambda^{1/2-\epsilon}}\left(\int\left|\int\sum_{|\nu-\nu'|\approx 2^r
\atop \nu'\leq \nu}\mathcal{Q}_{\lambda,\delta}^{\nu}f(x,t)
\overline{\mathcal{Q}_{\lambda,\delta}^{\nu'}f(x,t)}dt\right|^2dx\right)^{\frac14}\nonumber\\
&&\,\,\,\,\,\,\,\,\,\,\,\,+C_{\epsilon}\lambda^{\epsilon}\left\|\left(\sum_{\nu}|\mathcal{Q}_{\lambda,\delta}^{\nu}f|^2\right)^{\frac12}\right\|_{L^4(L^2)}.
\end{eqnarray}

We will need  further decompose based on $r$. For this purpose, we select $\rho\in C_0^{\infty}\left((-1,1)\right)$  satisfying the condition $\sum_{j\in\mathbb{Z}}\left(\rho(u-j)\right)^2\equiv 1$. Subsequently, we define, for $0<\epsilon_0 \ll \epsilon$,
\begin{equation}\label{general-gn2}
a_{\lambda,r}^{\nu,j}(\xi,t)=\left\{\begin{array}{ll}
\chi_{\nu}(\xi)a_{\lambda}(\xi,t)\rho\left(2^{-r}\lambda^{-1/2}q(\xi)-j\right),& \mbox{if $1\leq 2^r\leq \lambda^{\epsilon_0}$},\\
\chi_{\nu}(\xi)a_{\lambda}(\xi,t)\rho\left(2^{-r}\lambda^{-1/2+\epsilon_0}q(\xi)-j\right),& \mbox{if $\lambda^{\epsilon_0}< 2^r\leq \lambda^{1/4}$},\\
\chi_{\nu}(\xi)a_{\lambda}(\xi,t)\rho\left(2^{r}\lambda^{-1}q(\xi)-j\right),& \mbox{if $\lambda^{1/4}< 2^r\leq \lambda^{1/2-\epsilon}$},
\end{array}\right.
\end{equation}
and define the operators
$$
\mathcal{Q}_{\lambda,\delta,r}^{\nu,j}f(x,t)=\int_{\mathbb{R}^2}e^{i\phi(x,t,\xi,\delta)}a_{\lambda,r}^{\nu,j}(\xi,t)
\widehat{f}(\xi)d\xi.
$$

\begin{lemma}\label{lemma-081701}
There is an constant $C_1$, independent of $\lambda$ and $\delta$ such that $|j-j'|\geq C_1$,
\begin{eqnarray}\label{general-gn2}
&&\left|\int
\mathcal{Q}_{\lambda,\delta,r}^{\nu,j}f(x,t)\overline{\mathcal{Q}_{\lambda,\delta,r}^{\nu',j'}f(x,t)}dxdt\right|\nonumber\\
&&\,\,\,\,\,\,\,\,\,\,\,\leq
\left\{\begin{array}{ll}
C_N\lambda^{2}2^{2r}[2^{r}|j-j'|\delta\lambda^{1/2}]^{-N}\|f\|_{L^1}^2,& \mbox{if $1\leq 2^r\leq \lambda^{\epsilon_0}$},\\
C_N\lambda^{2-2\epsilon_0}2^{2r}[2^{r}|j-j'|\delta\lambda^{1/2-\epsilon_0}]^{-N}\|f\|_{L^1}^2,& \mbox{if $\lambda^{\epsilon_0}< 2^r\leq \lambda^{1/4}$},\\
C_N\lambda^32^{-2r}[(\lambda\delta)2^{-r}|j-j'|]^{-N}\|f\|_{L^1}^2,& \mbox{if $\lambda^{1/4}< 2^r\leq \lambda^{1/2-\epsilon}$}.
\end{array}\right.
\end{eqnarray}
\end{lemma}
\begin{proof}
To simplify the proof, we will only consider the case where $\lambda^{1/4}< 2^r\leq \lambda^{1/2-\epsilon}$.
It is worth noting that if $|j-j'|>4$, then the inequality $$|\phi_t'(x,t,\xi,\delta)-\phi_t'(x,t,\xi',\delta)|\geq \delta|q(\xi)-q(\xi')|-\delta^2|R_t'(\xi)-R_t'(\xi')|\geq c_0|j-j'|(\lambda\delta)2^{-r}$$
holds without any restrictions on $\nu$ and $\nu'$.

We observe that the absolute value of (\ref{general-gn2}) on the left-hand side is equivalent
 to
\begin{equation}
\int\mathcal{H}_{\nu,\nu'}^{j,j'}(x,\xi,\xi')\widehat{f}(\xi)\overline{\widehat{f}(\xi')}d\xi d\xi'dx
\end{equation}
where
\begin{equation}\label{0817000002}
\mathcal{H}_{\nu,\nu'}^{j,j'}(x,\xi,\xi')=\int e^{iH(x,t,\xi,\xi',\delta)}b_{j,j'}^{\nu,\nu'}(\xi,\xi',t)dt
\end{equation}
with
\begin{equation}
H(x,t,\xi,\xi',\delta)=\phi(x,t,\xi,\delta)-\phi(x,t,\xi',\delta)
\end{equation}
and
\begin{equation}
b_{j,j'}^{\nu,\nu'}(\xi,\xi',t)=a_{\lambda,r}^{\nu,j}(\xi,t)\overline{a_{\lambda,r}^{\nu',j'}(\xi',t)}.
\end{equation}

We evaluate (\ref{0817000002}) by integration by part argument.
If $\mathcal{D}$ denote the adjoint of operator $\frac{\partial_t}{iH_{t}'}$,
then we have
\begin{equation}
\mathcal{H}_{\nu,\nu'}^{j,j'}(x,\xi,\xi')=\int e^{iH(x,t,\xi,\xi',\delta)}\mathcal{D}^Nb_{j,j'}^{\nu,\nu'}(\xi,\xi',t)dt.
\end{equation}
It is easy to check that
\begin{equation}
|\mathcal{D}^Nb_{j,j'}^{\nu,\nu'}(\xi,\xi',t)|\leq C_{N} (|j-j'|(\lambda\delta)2^{-r})^{-N}.
\end{equation}
Thus we have
\begin{equation}
|\mathcal{H}_{\nu,\nu'}^{j,j'}(x,\xi,\xi')|\leq C_N  (|j-j'|(\lambda\delta)2^{-r})^{-N}.
\end{equation}

Therefore, the integral of (\ref{general-gn2}) is $\leq C_{N} \lambda^32^{-2r}(|j-j'|\delta\lambda2^{-r})^{-N}\|f\|_{L^1}^2$.
\end{proof}
By the decompositions (\ref{general-gn2}),
we have
\begin{eqnarray}\label{equ-0920-002}
\|\mathcal{Q}_{\lambda,\delta}f\|_{L^4(L^2)}&\leq&
\sum_{\lambda^{\epsilon}\leq2^r\leq \lambda^{1/2-\epsilon}}\left(\int\left|\int\sum_{|j-j'|\leq C_1}\sum_{|\nu-\nu'|\approx 2^r
\atop \nu'\leq \nu}\mathcal{Q}_{\lambda,\delta}^{\nu}f(x,t)
\overline{\mathcal{Q}_{\lambda,\delta}^{\nu'}f(x,t)}dt\right|^2dx\right)^{\frac14}\nonumber\\
&&\leq \sum_{\lambda^{\epsilon}\leq2^r\leq \lambda^{1/2-\epsilon}}\left(\int\left|\int\sum_{|j-j'|> C_1}\sum_{|\nu-\nu'|\approx 2^r
\atop \nu'\leq \nu}\mathcal{Q}_{\lambda,\delta}^{\nu}f(x,t)
\overline{\mathcal{Q}_{\lambda,\delta}^{\nu'}f(x,t)}dt\right|^2dx\right)^{\frac14}\nonumber\\
&&+C_{\epsilon}\lambda^{\epsilon}\left\|\left(\sum_{\nu}|\mathcal{Q}_{\lambda,\delta}^{\nu}f|^2\right)^{\frac12}\right\|_{L^4(L^2)}.
\end{eqnarray}

Using Lemma \ref{lemma-081701}, we see that the second term is bounded by $C_N \lambda^{-N}\|f\|_{L^1}^2$  for $|j-j'|\geq C_1$ and $\lambda\geq \delta^{-2}$.
If now we apply the Schwarz inequality with respect to the $t$ variable, we obtain
\begin{eqnarray}\label{2024082302}
\|\mathcal{Q}_{\lambda,\delta}f\|_{L^4{(L^2)}}\leq C_{\epsilon}\lambda^{\epsilon}\left\|\left(\sum_{\nu}|\mathcal{Q}_{\lambda,\delta}^{\nu}f|^2\right)^{1/2}\right\|_{L^4}
+C_{\epsilon,N}\lambda^{-N}\|f\|_{L^4}+C_{\epsilon}(I_r)^{1/4}
\end{eqnarray}
where
\begin{eqnarray}
I_r&=&\int\int\left|\sum_{|j-j'|\leq C_1}\sum_{|\nu-\nu'|\approx 2^r\atop \nu'\leq \nu}
\mathcal{Q}_{\lambda,\delta,r}^{\nu,j}f(x,t)\overline{\mathcal{Q}_{\lambda,\delta,r}^{\nu',j'}f(x,t)}\right|^2dtdx\nonumber\\
&=&\int\int\sum_{|j-j'|\leq C_1}\sum_{|k-k'|\leq C_1}
\sum_{|\nu-\nu'|\approx 2^r\atop \nu'\leq \nu}\sum_{|\mu-\mu'|\approx 2^r\atop \mu\leq \mu'}
\mathcal{Q}_{\lambda,\delta,r}^{\nu,j}f(x,t)\mathcal{Q}_{\lambda,\delta,r}^{\mu,k}f(x,t)\nonumber\\
&&\times
\overline{\mathcal{Q}_{\lambda,\delta,r}^{\nu',j'}f(x,t)}\overline{\mathcal{Q}_{\lambda,\delta,r}^{\mu',k'}f(x,t)}dtdx.
\end{eqnarray}

In order to estimate the term $I_r$ we have to use another orthogonality estimate.

\begin{lemma}\label{lemma-0817}(see Lemma 6.8 in \cite{1993-MSS-loacal-smoothing})
Let $\Psi\in C^4(\mathbb{R}^2)$ be homogenous of degree one.
Let $\gamma<\frac{\pi}{4}$, $a_0\leq 1/4$, $A_0\geq1$.
Let $\mathcal{S}_{\lambda}$ be the intersection of a sector which  subtends an angle of size $\gamma$ with the
annulus $\{\xi:(1-a_0)\lambda \leq |\xi|\leq (1+a_0)\lambda\}$.
Let $h\in C^1(\mathbb{R}^2\backslash 0)$ be homogenous of degree one such that
$b_0|\xi|\leq h(\xi)\leq b_1|\xi|$, $|\nabla_{\xi} h(\xi)|\leq b_2$ for some positive constants $b_0$, $b_1$, $b_2$.

Suppose that $a_0^{-1}\leq 2^{n}$, $n\leq l$, $2^l\leq \gamma \lambda^{1/2}$
and $\xi$, $\eta$, $\xi'$, $\eta'\in\mathcal{S}_{\lambda}$ are chosen such that for given integers $\nu$, $\mu$, $\nu'$, $\mu'$

(1) $|\arg(\xi)-\nu\lambda^{-1/2}|\leq \lambda^{-1/2}$; $|\arg(\xi')-\nu'\lambda^{-1/2}|\leq \lambda^{-1/2}$;

(2) $|\arg(\eta)-\mu\lambda^{-1/2}|\leq \lambda^{-1/2}$; $|\arg(\eta')-\mu'\lambda^{-1/2}|\leq \lambda^{-1/2}$;

(3) $2^{l-1}\lambda^{-1/2}\leq \max\{|\arg(\xi)-\arg(\eta)|,|\arg(\xi')-\arg(\eta')|\}\leq 2^{l+1}\lambda^{-1/2}$;

(4) $|h(\xi)-h(\xi'))|\leq 2^{-n}\lambda$;

(5) $|h(\eta)-h(\eta'))|\leq 2^{-n}\lambda$.\\

Then one can choose $\gamma$, $a_0$ sufficiently small, and $A_0$ sufficiently large
(only depending on $\Psi$, $b_0$, $b_1$, $b_2$)
such that for all $\nu$, $\mu$, $\nu'$, $\mu'$
with $A_02^{l-n}\leq |\nu-\nu'|+|\mu-\mu'|\leq \gamma \lambda^{1/2}$,
\begin{eqnarray}\label{ineq-000001}
&&|\Psi(\xi)+\Psi(\eta)-\Psi(\xi')-\Psi(\eta')|\nonumber\\
&&\,\,\,\,\,\,\,\,\leq C[(2^l|\mu-\mu'|+|\mu-\mu'|^2)+(2^l|\nu-\nu'|+|\nu-\nu'|^2)+|\xi+\eta-\xi'-\eta'|].
\end{eqnarray}
Suppose now that $\Psi$ satisfies the additional assumption rank $\Psi_{\xi\xi}''=1$.
Then if either $\mu\leq\nu$ and  $\mu'\leq\nu'$ or $\nu\leq\mu$ and  $\nu'\leq\mu'$
and if $A_02^{l-n}\leq |\nu-\nu'|+|\mu-\mu'|\leq \gamma \lambda^{1/2} $,
we have also with suitable positive constants $c_0$, $C_0$
\begin{eqnarray}\label{ineq-000002}
&&|\Psi(\xi)+\Psi(\eta)-\Psi(\xi')-\Psi(\eta')|\nonumber\\
&&\,\,\,\,\,\,\,\,\geq c_0[(2^l|\nu-\nu'|+|\nu-\nu'|^2)+(2^l|\mu-\mu'|+|\mu-\mu'|^2)]-C_0|\xi+\eta-\xi'-\eta'|.
\end{eqnarray}

We also get for all $\nu$, $\mu$, $\nu'$, $\mu'$
with $A_02^{l-n}\leq |\nu-\mu'|+|\mu-\nu'|\leq \gamma \lambda^{1/2}$,
\begin{eqnarray}\label{ineq-000001-2}
&&|\Psi(\xi)+\Psi(\eta)-\Psi(\xi')-\Psi(\eta')|\nonumber\\
&&\,\,\,\,\,\,\,\,\leq C[(2^l|\nu-\mu'|+|\mu-\nu'|^2)+(2^l|\nu-\mu'|+|\mu-\nu'|^2)+|\xi+\eta-\xi'-\eta'|].
\end{eqnarray}
Suppose now that $\Psi$ satisfies the additional assumption rank $\Psi_{\xi\xi}''=1$.
Then if either $\mu\leq\nu$ and  $\nu'\leq\mu'$ or $\nu\leq\mu$ and  $\mu'\leq\nu'$
and if $A_02^{l-n}\leq |\nu-\mu'|+|\mu-\nu'|\leq \gamma \lambda^{1/2} $,
we have also with suitable positive constants $c_0$, $C_0$
\begin{eqnarray}\label{ineq-000002-2}
&&|\Psi(\xi)+\Psi(\eta)-\Psi(\xi')-\Psi(\eta')|\nonumber\\
&&\,\,\,\,\,\,\,\,\geq c_0[(2^l|\nu-\mu'|+|\mu-\nu'|^2)+(2^l|\nu-\mu'|+|\mu-\nu'|^2)]-C_0|\xi+\eta-\xi'-\eta'|.
\end{eqnarray}
\end{lemma}

\begin{lemma}\label{2014082301}
Suppose that $|\nu-\nu'|\approx 2^r$, $|\mu-\mu'|\approx 2^r$, $\nu'\leq\nu$,
$\mu\leq\mu'$, $|j-j'|\leq C_1$, $|k-k'|\leq C_1$,
or $\lambda>A_1$ for sufficient large $A_1$.
Furthermore assume that $|\nu-\mu|+|\nu'-\mu'|\geq A_2$
and $|\mu-\nu'|+|\nu-\mu'|\geq A_2$ for a sufficiently large  constant $A_2$,
independent of $\lambda$.

Suppose  $\lambda^{\epsilon}\leq 2^r\leq \lambda^{1/4}$.
If $\mu<\nu$ and $\nu'<\mu'$ then
\begin{eqnarray}\label{ineq-2fenjie-03-01}
&&\left|\int
\mathcal{Q}_{\lambda,\delta,r}^{\nu,j}f(x,t)\mathcal{Q}_{\lambda,\delta,r}^{\mu,k}f(x,t)
\overline{\mathcal{Q}_{\lambda,\delta,r}^{\nu',j'}f(x,t)}\overline{\mathcal{Q}_{\lambda,\delta,r}^{\mu',k'}f(x,t)}dtdx\right|\nonumber\\
&&\,\,\,\,\,\,\leq C_N\lambda^4 2^{4r}(2^r\delta|\nu-\mu'|+2^r\delta|\mu-\nu'|)^{-N}\|f\|_{L^1}^4;
\end{eqnarray}
if $\nu<\mu$ or $\mu'<\nu'$ then
\begin{eqnarray}\label{ineq-2fenjie-03-02}
&&\left|\int
\mathcal{Q}_{\lambda,\delta,r}^{\nu,j}f(x,t)\mathcal{Q}_{\lambda,\delta,r}^{\mu,k}f(x,t)
\overline{\mathcal{Q}_{\lambda,\delta,r}^{\nu',j'}f(x,t)}\overline{\mathcal{Q}_{\lambda,\delta,r}^{\mu',k'}f(x,t)}dtdx\right|\nonumber\\
&&\,\,\,\,\,\,\leq C_N\lambda^4 2^{4r}(2^r\delta|\nu-\mu|+2^r\delta|\mu'-\nu'|)^{-N}\|f\|_{L^1}^4.
\end{eqnarray}

Suppose $\lambda^{1/4}\leq 2^r\leq \lambda^{\frac12-\epsilon}$.
If $\mu<\nu$ and $\nu'<\mu'$ then
\begin{eqnarray}\label{ineq-2fenjie-01-03}
&&\left|\int
\mathcal{Q}_{\lambda,\delta,r}^{\nu,j}f(x,t)\mathcal{Q}_{\lambda,\delta,r}^{\mu,k}f(x,t)
\overline{\mathcal{Q}_{\lambda,\delta,r}^{\nu',j'}f(x,t)}\overline{\mathcal{Q}_{\lambda,\delta,r}^{\mu',k'}f(x,t)}dtdx\right|\nonumber\\
&&\,\,\,\,\,\,\leq \lambda^6 2^{-4r}(2^r\delta|\nu-\mu'|+2^r\delta|\mu-\nu'|)^{-N}\|f\|_{L^1}^4;
\end{eqnarray}
if $\nu<\mu$ or $\mu'<\nu'$ then
\begin{eqnarray}\label{ineq-2fenjie-02-04}
&&\left|\int
\mathcal{Q}_{\lambda,\delta,r}^{\nu,j}f(x,t)\mathcal{Q}_{\lambda,\delta,r}^{\mu,k}f(x,t)
\overline{\mathcal{Q}_{\lambda,\delta,r}^{\nu',j'}f(x,t)}\overline{\mathcal{Q}_{\lambda,\delta,r}^{\mu',k'}f(x,t)}dtdx\right|\nonumber\\
&&\,\,\,\,\,\,\leq \lambda^6 2^{-4r}(2^r\delta|\nu-\mu|+2^r\delta|\mu'-\nu'|)^{-N}\|f\|_{L^1}^4.
\end{eqnarray}

\end{lemma}
\begin{proof}
One notice that the left side of (\ref{ineq-2fenjie-03-01}) is equivalent to the absolute value of
\begin{equation}
\int\mathcal{H}_{j,k,j',k'}^{\nu,\mu,\nu',\mu',\delta}(x,\xi,\eta,\xi',\eta')\widehat{f}(\xi)\widehat{f}(\eta)\overline{\widehat{f}(\xi')}
\overline{\widehat{f}(\eta')}
d\xi d\xi'd\eta d\eta' dx
\end{equation}
where
\begin{equation}
\mathcal{H}_{j,k,j',k'}^{\nu,\mu,\nu',\mu',\delta}(x,\xi,\eta,\xi',\eta')=
\int e^{i(\Phi(x,t,\xi,\eta,\xi',\eta',\delta))}b_{j,k,j',k'}^{\nu,\mu,\nu',\mu'}(t,\xi,\eta,\xi',\eta')dt
\end{equation}
with
$$
\Phi(x,t,\xi,\eta,\xi',\eta',\delta)=\phi(x,t,\xi,\delta)+\phi(x,t,\eta,\delta)-\phi(x,t,\xi',\delta)-\phi(x,t,\eta',\delta)
$$
and
$$
b_{j,k,j',k'}^{\nu,\mu,\nu',\mu'}(t,\xi,\eta,\xi',\eta')=a_{\lambda,r}^{\nu,j}(\xi,t)a_{\lambda,r}^{\mu,k}(\eta,t)
\overline{a_{\lambda,r}^{\nu',j'}(\xi',t)}\overline{a_{\lambda,r}^{\mu',k'}(\eta',t)}.
$$
We first observe that
\begin{equation}\label{equ-2024082301}
|\partial_t^{\alpha}b_{j,k,j',k'}^{\nu,\mu,\nu',\mu'}(t,\xi,\eta,\xi',\eta')|\leq C_{\alpha}.
\end{equation}

If $\mathcal{L}$ denote the adjoint of operator $\frac{\partial_t}{i\Phi_{t}'}$,
then we have
\begin{equation}\label{ineq-qingxing4-1}
\mathcal{H}_{j,k,j',k'}^{\nu,\mu,\nu',\mu',\delta}(x,\xi,\eta,\xi',\eta')=
\int e^{i(\Phi(x,t,\xi,\eta,\xi',\eta',\delta))}\mathcal{L}^Nb_{j,k,j',k'}^{\nu,\mu,\nu',\mu'}(t,\xi,\eta,\xi',\eta')dt.
\end{equation}
A straightforward computation shows that $\mathcal{L}^Nb_{j,k,j',k'}^{\nu,\mu,\nu',\mu'}(t,\xi,\eta,\xi',\eta')$ is the sum of the form
\begin{eqnarray}\label{ineq-chang-fenmu0000}
&&\nonumber\frac{\prod_{i=1}^{N_2}\left|\delta^2R_t^{(n_i)}(\xi,t,\delta)+\delta^2R_t^{(n_i)}(\eta,t,\delta)-\delta^2R_t^{(n_i)}(\xi',t,\delta)-\delta^2R_t^{(n_i)}(\eta',t,\delta)\right|}{
|\delta \tilde{q}_t'(\xi,t,\delta)+\delta \tilde{q}_t'(\eta,t,\delta)-\delta \tilde{q}_t'(\xi',t,\delta)-\delta \tilde{q}_t'(\eta',t,\delta)|^{N+N_1}}\\
&&\,\,\,\,\,\,\,\,\times |\partial_t^{\alpha}b_{j,k,j',k'}^{\nu,\mu,\nu',\mu'}(t,\xi,\eta,\xi',\eta')|,
\end{eqnarray}
where
$0\leq N_1\leq N$, $0\leq N_2\leq N_1$, $n_i\geq 2$ and $n_1+\cdots+n_{N_2}+\alpha=N+N_1$.

The worst case to bound the numerator of (\ref{ineq-chang-fenmu0000}) is that we take $N_1=N_2=0$.
In order to use Lemma \ref{lemma-0817} to estimate the numerator of (\ref{ineq-chang-fenmu0000}),
we will replace $\gamma$ by $\lambda^{-\epsilon}$ and $h(\xi)$ by $\partial_tR(\xi,t,\delta)$.
In order to estimate the denominator of (\ref{ineq-chang-fenmu0000}) from below,
we will replace $\gamma$ by $\lambda^{-\epsilon}$ and $h(\xi)$ by $q(\xi)$.

If $\nu<\mu$ or $\mu'<\nu'$, then $\nu'<\nu<\mu<\mu'$ or $\mu<\mu'<\nu'<\nu$ respectively.
If $\lambda^{\epsilon}<2^r\leq \lambda^{1/4}$ and $\xi\in \supp a_{\lambda,r}^{\nu,j}$, $\xi'\in \supp a_{\lambda,r}^{\nu',j'}$,
then
\begin{equation}
|h(\xi)-h(\xi')|\leq C 2^{r}\lambda^{1/2-\epsilon_0}.
\end{equation}
So if we let $2^{-n}\lambda$ equal the right-hand side,
we have $2^{l-n}\leq C 2^{r}2^l\lambda^{-1/2-\epsilon_0} \leq A_0^{-1} (|\nu-\nu'|+|\mu-\mu'|)$.
The same argument applies $|h(\eta)-h(\eta')|$.

If $\lambda^{1/4}<2^r\leq \lambda^{1/2-\epsilon}$ and $\xi\in \supp a_{\lambda,r}^{\nu,j}$, $\xi'\in \supp a_{\lambda,r}^{\nu',j'}$,
then
\begin{equation}
|h(\xi)-h(\xi')|\leq C 2^{-r}\lambda.
\end{equation}
So if we let $2^{-n}\lambda$ equal the right-hand side,
we have $2^{l-n}\leq C$.
The same argument applies $|h(\eta)-h(\eta')|$.
Thus, we  can use (\ref{ineq-000001-2}), (\ref{ineq-000002-2}) and (\ref{equ-2024082301}) for (\ref{ineq-chang-fenmu0000})
to get
\begin{equation*}
|\mathcal{L}^Nb_{j,k,j',k'}^{\nu,\mu,\nu',\mu'}(t,\xi,\eta,\xi',\eta')|\leq C
\frac{1}{\left(\delta|\nu-\nu'|(2^l+|\nu-\nu'|)+\delta|\mu-\mu'|(2^l+|\mu-\mu'|)\right)^{N}}.
\end{equation*}
Since $\mu<\nu'<\nu<\mu'$, $\mu<\nu'<\mu'<\nu$, $\nu'<\mu<\mu'<\nu$, or $\nu'<\mu<\nu<\mu'$ and
$|\nu-\mu|+|\nu'-\mu'|\approx 2^l$, we can obtain
\begin{equation*}
|\mathcal{L}^Nb_{j,k,j',k'}^{\nu,\mu,\nu',\mu'}(t,\xi,\eta,\xi',\eta')|\leq C
\frac{1}{\left(2^r\delta|\mu-\nu'|+2^r\delta|\mu'-\nu|\right)^N}.
\end{equation*}
This inequality implies (\ref{ineq-2fenjie-03-01}) and (\ref{ineq-2fenjie-01-03}).

Let us assume $|\nu-\mu|+|\nu'-\mu'|\approx 2^l$.
If $\mu<\nu$ and $\nu'<\mu'$, then
we have one of the following cases: $\mu<\nu'<\nu<\mu'$, $\mu<\nu'<\mu'<\nu$, $\nu'<\mu<\mu'<\nu$, or $\nu'<\mu<\nu<\mu'$.
In either one of the these cases, we notice that $2^l\approx 2^r$.

If $\lambda^{\epsilon}<2^r\leq \lambda^{1/4}$ and $\xi\in \supp a_{\lambda,r}^{\nu,j}$, $\xi'\in \supp a_{\lambda,r}^{\nu',j'}$,
then
\begin{equation}
|h(\xi)-h(\xi')|\leq C 2^{r}\lambda^{1/2-\epsilon_0}.
\end{equation}
So if we let $2^{-n}\lambda$ equal the right-hand side,
we have $2^{l-n}\leq C 2^{2r}\lambda^{-1/2-\epsilon_0} \leq C_0$.
The same argument applies $|h(\eta)-h(\eta')|$.

If $\lambda^{1/4}<2^r\leq \lambda^{1/2-\epsilon}$ and $\xi\in \supp a_{\lambda,r}^{\nu,j}$, $\xi'\in \supp a_{\lambda,r}^{\nu',j'}$,
then
\begin{equation}
|h(\xi)-h(\xi')|\leq C 2^{-r}\lambda.
\end{equation}
So if we let $2^{-n}\lambda$ equal the right-hand side,
we have $2^{l-n}\leq C$.
The same argument applies $|h(\eta)-h(\eta')|$.
Thus, we  can use (\ref{ineq-000001-2}), (\ref{ineq-000002-2}) and (\ref{equ-2024082301}) for (\ref{ineq-chang-fenmu0000})
to get
\begin{equation*}
|\mathcal{L}^Nb_{j,k,j',k'}^{\nu,\mu,\nu',\mu'}(t,\xi,\eta,\xi',\eta')|\leq C
\frac{1}{\left(\delta|\nu-\nu'|(2^l+|\nu-\nu'|)+\delta|\mu-\mu'|(2^l+|\mu-\mu'|)\right)^{N}}.
\end{equation*}
Since
$|\nu-\mu|+|\nu'-\mu'|\approx 2^l$, we can obtain
\begin{equation*}
|\mathcal{L}^Nb_{j,k,j',k'}^{\nu,\mu,\nu',\mu'}(t,\xi,\eta,\xi',\eta')|\leq C
\frac{1}{\left(2^r\delta|\mu-\nu|+2^r\delta|\mu'-\nu'|\right)^N}.
\end{equation*}
This inequality implies (\ref{ineq-2fenjie-03-02}) and (\ref{ineq-2fenjie-02-04}).
\end{proof}

By utilizing the Schwarz inequality, it can be demonstrated that Lemma \ref{2014082301} implies
$$
I_r\leq C_{\epsilon}\lambda^{4\epsilon}\delta^{-2}\left\|\left(\sum_{\nu,j}|\mathcal{Q}_{\lambda,\delta}^{\nu,j}f|^2\right)^{1/2}\right\|_{L^4}^4
+C_{\epsilon,N}\lambda^{-N}\|f\|_{L^4}^{4}
$$
 and since $\epsilon$ can be chosen arbitrary small, this together with (\ref{2024082302}) proves (\ref{ineq-082102}).

%The main inequality we want to prove is
%\begin{eqnarray}\label{inequ-main-noniso}
%\|\mathcal{Q}_{\lambda,\delta}f\|_{L^4(\mathbb{R}^2,L^2([1,2]))}&\leq& C_{\epsilon,\epsilon'}\lambda^{\epsilon}\delta^{-\frac12-\epsilon'}
%\sum_{1\leq 2^r\leq \lambda^{1/2-\epsilon}}\left\|\left(\sum_{\nu,j}\left|\mathcal{Q}_{\lambda,\delta,r}^{\nu,j}f\right|^2\right)^{1/2}\right\|_{L^4}\nonumber\\
%&&+\,\,\,\,\,C_{\epsilon,\epsilon',N}\lambda^{-N}\|f\|_{L^4}.
%\end{eqnarray}

Thus it is enough for us to  show that
\begin{equation}\label{000001}
\left\|\left(\sum_{\nu,j}\left|\mathcal{Q}_{\lambda,\delta,r}^{\nu,j}f\right|^2\right)^{1/2}\right\|_{L^4}\leq C_{\epsilon}\lambda^{\epsilon}\|f\|_{L^4}.
\end{equation}

We only establish the validity of (\ref{000001}) for the scenario where $\lambda^{1/4}\leq 2^r\leq \lambda^{1/2-\epsilon}$, as the proof for the case $1\leq 2^r\leq \lambda^{1/4}$ is similar. We perform a decomposition of each sector $S^{\nu}=\supp_{\xi}\chi_{\nu}(\xi)$ that remains independent of $(x,t)$. Let $\rho$ be defined as mentioned above and

$$
\beta_{\lambda,r}^{\kappa}(\xi)=\rho\left(\lambda^{-1}2^{r}\langle \xi,\xi_{\nu}\rangle-\kappa\right)
$$
and define $Q_{\lambda,r}^{\nu,\kappa}$ as
$$
\widehat{Q_{\lambda,r}^{\nu,\kappa}f}=\chi_{S^{\nu}}(\xi)\beta_{\lambda,r}^{\kappa}(\xi)\widehat{f}(\xi).
$$

We will utilize a square function estimation proposed by  C$\acute{o}rdoba$ \cite{1982-cordoda}, which states,
\begin{equation}\label{inequ-coba001}
\left\|\left(\sum_{\nu,\kappa}|Q_{\lambda,r}^{\nu,\kappa}f|^2\right)^{\frac12}\right\|_{L^4(\mathbb{R}^2)}
\leq C\left(\log\lambda\right)^{c}\|f\|_{L^4(\mathbb{R}^2)}.
\end{equation}
Here, $C$ and $c$ are positive constants.
Our objective would be achieved if we can demonstrate
\begin{equation}\label{ineq-081901}
\left\|\left(\sum_{\nu,j}|\mathcal{Q}_{\lambda,\delta,r}^{\nu,j}f|^2\right)^{\frac12}\right\|_{L^4(\mathbb{R}^3)}
\leq C\lambda^{\epsilon}\left\|\left(\sum_{\nu,\kappa}|Q_{\lambda,r}^{\nu,\kappa}f|^2\right)^{\frac12}\right\|_{L^4(\mathbb{R}^3)}.
\end{equation}

To this end, set
$$
\mathcal{Q}_{\lambda,\delta,r}^{\nu,j}f=\sum_{\kappa}\mathcal{Q}_{\lambda,\delta,r}^{\nu,j,\kappa}Q_{\lambda,r}^{\kappa}f,
$$
where
$$
\mathcal{Q}_{\lambda,\delta,r}^{\nu,j,\kappa}f(x,t)=\int e^{i\phi(x,t,\xi,\delta)}a_{\lambda,\delta,r}^{\nu,j,\kappa}
(\xi,t,\delta)\beta_{\lambda,r}^{\kappa}(\xi)\widehat{f}(\xi)d\xi.
$$

Also note that if we denote the kernel of $\mathcal{Q}_{\lambda,\delta,r}^{\nu,j,\kappa}$ as $K_{\lambda,\delta,r}^{\nu,j,\kappa}$, there exist integers $\kappa(\nu,j)$ such that for any $x,t,y,$ and $\delta$, the inequality
\begin{equation}\label{ineq-081902001}
K_{\lambda,\delta,r}^{\nu,j,\kappa}(x,t,y,\delta)=0
\end{equation} holds true when $|\kappa-\kappa(\nu,j)|\geq C_0$, where $C_0$ is a uniform constant. Furthermore, by employing an integration by parts argument, it can be shown that the kernels satisfy
\begin{eqnarray}\label{ineq-082101}
|K_{\lambda,\delta,r}^{\nu,j,\kappa}(x,t,y,\delta)|&\leq& C_N\frac{2^{-r}\lambda}{(1+2^{-r}\lambda|\langle \tilde{q}'_{\xi}(\xi_{\nu},t,\delta)+x-y,\xi_{\nu}\rangle|)^N}\nonumber\\
&\times&\frac{\lambda^{1/2}}{\left(1+\lambda^{1/2}|\langle \tilde{q}_{\xi}'(\xi_{\nu},t,\delta)+x-y,\xi_{\nu}^{\bot} \rangle|\right)^{N}}
\end{eqnarray}
where $\tilde{q}(\xi,t,\delta)=E(\xi)+t\delta q(\xi)+\delta^2(\xi,t,\delta)$ and the unite vector $\xi_{\nu}^{\bot}$ is orthogonal to $\xi_{\nu}$.
In particular, we see that
\begin{equation} \label{inequ-081903}
\int \left|K_{\lambda,\delta,r}^{\nu,j,\kappa}(x,t,y)\right|dy\leq C.
\end{equation}

Using (\ref{ineq-081902001})  and (\ref{inequ-081903}), the  square of the left-hand side of
(\ref{ineq-081901})  is bounded by
\begin{eqnarray*}
&&\sup_{\|g\|_{L^2}=1}\left|\sum_{\nu,j,\kappa}\int_{\mathbb{R}^3}|\mathcal{Q}_{\lambda,\delta,r}^{\nu,j,\kappa}Q_{\lambda,r}^{\nu,\kappa}f(x,t)|^2
g(x,t)dxdt\right|\\
&&\,\,\,\,\leq C_{\epsilon}\lambda^{2\epsilon}\sup_{\|g\|_{L^2}=1}\int_{\mathbb{R}^2}\sum_{\nu,\kappa}|Q_{\lambda,r}^{\nu,\kappa}f|^2\\
&&\,\,\,\,\,\,\,\,\,\,\,\,\,\,\,\,\,\,\,\,\,\,\,\,\,\,\times\sup_{v}\left\{\int_{\mathbb{R}^3}\sup_{j,\kappa}|K_{\lambda,\delta,r}^{\nu,j,\kappa}(x,t,y)||g(x,t)|dxdt\right\}dy.
\end{eqnarray*}
Now the proof of (\ref{ineq-081901}) is complete once we have established the maximal inequality
$$
\left(\int_{\mathbb{R}^2}\sup_{v}\left|\int_{\mathbb{R}^3}\sup_{j,\kappa}|K_{\lambda,\delta,r}^{\nu,j,\kappa}(x,t,y)||g(x,t)|dxdt\right|^2dy\right)^{1/2}\leq C_{\epsilon}\lambda^{\epsilon}\|g\|_{L^2(\mathbb{R}^3)},
$$
but this is a consequence of (\ref{ineq-082101}) and Theorem 5.3 in \cite{1993-MSS-loacal-smoothing}.

\textbf{Declarations}\\
\textbf{Funding}:
The research was supported by the
Hebei Province Provincial Universities Basic Research Project Funding (Grant Nos. ZQK202305)\\
\textbf{Availability of data and material}\\
Not applicable\\
\textbf{Code availability}\\
Not applicable\\
%\textbf{Authors' contributions}\\
%All authors contributed equally and significantly in writing this
%paper.
%All authors read and approved the final manuscript.

\end{document}